\documentclass[10pt, a4paper, oneside]{amsart}

\makeatletter
\def\@settitle{\begin{center}%
		\baselineskip14\p@\relax
		\normalfont\LARGE\bfseries
		\@title
	\end{center}%
}

\def\section{\@startsection{section}{1}%
	\z@{.7\linespacing\@plus\linespacing}{.5\linespacing}%
	{\normalfont\large\bfseries}}

\def\subsection{\@startsection{subsection}{2}%
	\z@{.5\linespacing\@plus.7\linespacing}{.5\linespacing}%
	{\normalfont\bfseries}}

\def\@setauthors{%
  \begingroup
  \def\thanks{\protect\thanks@warning}%
  \trivlist
  \centering\footnotesize \@topsep30\p@\relax
  \advance\@topsep by -\baselineskip
  \item\relax
  \author@andify\authors
  \def\\{\protect\linebreak}%
  \authors%
  \ifx\@empty\contribs
  \else
    ,\penalty-3 \space \@setcontribs
    \@closetoccontribs
  \fi
  \endtrivlist
  \endgroup
}

\makeatother

\usepackage[usenames, dvipsnames]{xcolor}
\definecolor{darkblue}{rgb}{0.0, 0.0, 0.45}
\definecolor{darkgreen}{rgb}{0.0, 0.45, 0}
\usepackage[colorlinks	= true,
raiselinks	= true,
linkcolor	= darkblue, %MidnightBlue,
citecolor	= Mahogany,
urlcolor	= darkgreen,
pdfauthor	= {M.A.S. Kolarijani},
pdftitle	= {},
pdfkeywords	= {},
pdfsubject	= {},
plainpages	= false]{hyperref}

\allowdisplaybreaks
\pdfoutput=1
\date{\today}

\usepackage{dsfont,amsfonts,amssymb,amsmath,amsthm}
\usepackage{mathtools, mathrsfs}

\usepackage{graphicx}
\usepackage[font=small,labelfont=rm]{subcaption}
\usepackage[font=small,margin=10pt]{caption}

\usepackage{booktabs}
\usepackage{nicefrac}

\usepackage{algorithm,algorithmic}%,algorithmicx}
\usepackage{enumitem}
\usepackage{tikz}

\theoremstyle{plain}
\newtheorem{Thm}{Theorem}[section]

\newtheorem{Prop}[Thm]{Proposition}

\newtheorem{Lem}[Thm]{Lemma}
\newtheorem{Cor}[Thm]{Corollary}
\newtheorem{As}[Thm]{Assumption}

\newtheorem{Rem}[Thm]{Remark}

\newtheorem{Set}{Setting}

\newcommand{\R}{\mathbb{R}}
\newcommand{\Ru}{\overline {\R}}

\newcommand{\Rlu}{\overline{\underline {\R}}}
\newcommand{\N}{\mathbb{N}}

\newcommand{\ra}{\rightarrow}

\newcommand{\ol}[1]{\overline{#1}}
\newcommand{\ul}[1]{\underline{#1}}
\newcommand{\wt}{\widetilde}

\newcommand{\Let}{\coloneqq}
\newcommand{\teL}{\eqqcolon}

\DeclareMathOperator*{\argmin}{\arg\!\min}
\DeclareMathOperator*{\argmax}{\arg\!\max}

\newcommand{\tr}{^{\top}}
\newcommand{\itr}{^{-\top}}

\newcommand{\norm}[1]{\left\Vert #1 \right\Vert}
\newcommand{\inner}[2]{\left\langle #1, #2 \right\rangle }

\def\ssum{\begingroup\textstyle \sum\endgroup}

\newcommand{\opt}{^\star}

\newcommand{\EE}{\mathds{E}}

\DeclareMathOperator{\sgn}{sgn}

\newcommand{\setc}[1]{\mathbb{#1}}
\newcommand{\setd}[1]{\mathbb{#1}^{\mathrm{d}}}
\newcommand{\setg}[1]{\mathbb{#1}^{\mathrm{g}}}
\newcommand{\setsg}[1]{\mathbb{#1}^{\mathrm{g}}_{\mathrm{sub}}}

\newcommand{\diam}[1]{\Delta_{#1}}

\newcommand{\lftc}[1]{#1^{*}}
\newcommand{\lftcd}[1]{#1^{*\mathrm{d}}}
\newcommand{\lftd}[1]{#1^{\mathrm{d}*}}
\newcommand{\lftdd}[1]{#1^{\mathrm{d}*\mathrm{d}}}
\newcommand{\bcc}[1]{#1^{**}}

\newcommand{\bcdd}[1]{#1^{\mathrm{d}*\mathrm{d}*}}
\newcommand{\lerp}[1]{\wt{#1}}
\newcommand{\llerp}[1]{\ol{#1}}
\newcommand{\ef}[1]{#1_{\mathrm{w}}}
\newcommand{\disc}[1]{#1^{\mathrm{d}}}

\newcommand{\dyn}{f}
\newcommand{\dynx}{f_\mathrm{s}}
\newcommand{\dynu}{f_\mathrm{i}}
\newcommand{\cost}{C}
\newcommand{\costx}{C_\mathrm{s}}
\newcommand{\costu}{C_\mathrm{i}}

\newcommand{\dpo}{\mathcal{T}}
\newcommand{\ddpo}{\mathcal{T}^{\mathrm{d}}}
\newcommand{\idpo}{\mathcal{T}^{\mathrm{i}}}
\newcommand{\iidpo}{\mathcal{I}}

\newcommand{\cdpo}{\widehat{\mathcal{T}}}
\newcommand{\dcdpo}{\widehat{\mathcal{T}}^{\mathrm{d}}}
\newcommand{\mdcdpo}{\widehat{\mathcal{T}}^{\mathrm{d}}_{\mathrm{m}}}
\newcommand{\idcdpo}{\widehat{\mathcal{I}}^{\mathrm{d}}}

\DeclareMathOperator{\dist}{d}

\DeclareMathOperator{\dish}{d_H}
\DeclareMathOperator{\ord}{\mathcal{O}}
\DeclareMathOperator{\co}{co}
\DeclareMathOperator{\lip}{L}
\DeclareMathOperator{\dom}{dom}

\usepackage[square,numbers]{natbib}

\graphicspath{{Fig/}}

\title[]{Fast Approximate Dynamic Programming \\
for Input-Affine Dynamics}

\author[]{M.~A.~S.~Kolarijani and P.~Mohajerin Esfahani}
\thanks{The authors are with Delft Center for Systems and Control, Delft University of Technology, 
Delft, The Netherlands. Email: \texttt{\{M.A.SharifiKolarijani, P.MohajerinEsfahani\}@tudelft.nl}.}
\thanks{The authors would like to thank G. F. Max for several fruitful discussions.}
\thanks{This research is part of a project that has received funding from the European Research Council (ERC) under the grant TRUST-949796.}
\begin{document}

\begin{abstract}
We propose two novel numerical schemes for approximate implementation of the dynamic programming~(DP) operation concerned with finite-horizon, optimal control of discrete-time systems with input-affine dynamics. 
The proposed algorithms involve discretization of the state and input spaces and are based on an alternative path that solves the dual problem corresponding to the DP operation. 
We provide error bounds for the proposed algorithms, along with a detailed analysis of their computational complexity. 
In particular, for a specific class of problems with separable data in the state and input variables, the proposed approach can reduce the typical time complexity of the DP operation from $\ord(XU)$ to $\ord (X+U)$, 
where $X$ and $U$ denote the size of the discrete state and input spaces, respectively. 
This reduction is achieved by an algorithmic transformation of the minimization in the DP operation to an addition via discrete conjugation. 

\end{abstract}

\maketitle

%===============================================================================
\section{Introduction}
\label{sec:intro}
%===============================================================================

Dynamic programming (DP) is one of the most common tools used for tackling sequential decision problems with applications in, e.g., optimal control, operation research, and reinforcement learning. 
The basic idea of DP is to solve the Bellman equation
\begin{equation} \label{eq:BE}
J_t(x_t)  = \min_{u_t} \big\{ \cost(x_t,u_t) + J_{t+1}(x_{t+1}) \big\},
\end{equation}
backward in time $t$ for the costs-to-go~$J_t$, where $\cost(x_t,u_t)$ is the cost of taking the control action~$u_t$ at the state~$x_t$ (value iteration). 
Arguably, the most important drawback of DP is in its high computational cost in solving problems with a large scale \emph{finite} state space, which are usually described as Markov decision processes (MDPs). 
Indeed, in \cite{balaji18}, the authors show that for a finite-horizon MDP, the problem of determining whether a control action $u_0$ is an optimal action at a given initial state $x_0$ using value iteration is EXPTIME-complete. 
For problems with a \emph{continuous} state space, solving the Bellman equation requires solving an infinite number of optimization problems. 
This usually renders the exact implementation of the DP operation impossible, except for a few cases with an available closed-form solution, e.g., linear quadratic regulator \cite[Sec.~4.1]{Bertsekas05}. 
To address this issue, various schemes have been introduced, commonly known as \emph{approximate} dynamic programming; see, e.g., \cite{Bertsekas19,Pow11}. 
A common scheme is to use a sample-based approach accompanied by some form of function approximation. 
This usually amounts to deploying a brute force search over the discretizations/abstractions of the state and input spaces, 
leading to a time complexity of at least $\ord(XU)$, where $X$ and $U$ are the cardinality of the discrete state and input spaces, respectively. 

For some DP problems, it is possible to reduce this complexity by using duality, i.e., approaching the minimization problem in~\eqref{eq:BE} in the conjugate domain. 
For instance, for the deterministic linear dynamics~$x_{t+1} = Ax_t + B u_t$ with the separable cost~$\cost(x_t,u_t) = \costx(x_t) + \costu(u_t)$, we have
\begin{equation} \label{eq:BE conj}
J_t(x_t)  \geq \costx(x_t) + \lftc{\left[ \lftc{\costu}(-B\tr \cdot) + \lftc{J_{t+1}} \right]} (Ax_t) ,
\end{equation}
where the operator~$\lftc{[\cdot]}$ denotes the Legendre-Fenchel transform, also known as (convex) conjugate transform. 
Under some  technical assumptions (including, among others, convexity of the functions~$\costu$ and $J_{t+1}$), we have equality in~\eqref{eq:BE conj}; see \cite[Prop.~5.3.1]{Bertsekas09}. 
Notice how the minimization operator in~\eqref{eq:BE} transforms to a simple addition in~\eqref{eq:BE conj}. 
This observation signals the possibility of a significant reduction in the time complexity of solving the Bellman equation, at least for particular classes of DP problems.    

Approaching the DP problem through the lens of the conjugate duality goes back to Bellman \cite{Bell62}. 
Applications of this idea for reducing the computational complexity were later explored in \cite{Esog90, Klein91}. 
Fundamentally, these approaches exploit the operational duality of infimal convolution and addition with respect to (w.r.t.) the conjugate transform: For two functions~$f_1,f_2: \R^n \ra [-\infty,  +\infty]$, we have $\lftc{(f_1 \Box f_2)} = \lftc{f_1} + \lftc{f_2}$, where 
\begin{equation}\label{eq:inf conv}
f_1 \Box f_2 (w) \Let \inf_{w_1,w_2} \{ f_1(w_1) + f_2(w_2): w_1+w_2 = w \},
\end{equation}
is the infimal convolution of $f_1$ and $f_2$ \cite{Rock74}. 
This is analogous to the well-known operational duality of convolution and multiplication w.r.t. the Fourier transform. 
Actually, the Legendre-Fenchel transform plays a similar role as the Fourier transform when the underlying algebra is the max-plus algebra, as opposed to the conventional plus-times algebra. 
Much like the extensive application of the latter operational duality upon introduction of the fast Fourier transform, ``fast'' numerical algorithms for conjugate transform can facilitate efficient applications of the former one. 
Interestingly, the first fast algorithm for computing (discrete) conjugate functions, known as fast Legendre transform, was inspired by fast Fourier transform, and enjoys the same \emph{log-linear} complexity in the number of data points; see \cite{Corrias96,Lucet96} and the references therein. 
Later, this complexity was reduced by introducing a \emph{linear-time} algorithm known as linear-time Legendre transform (LLT)~\cite{Lucet97}. 
We refer the interested reader to~\cite{Lucet10} for an extensive review of these algorithms (and other similar algorithms) and their applications. 
In this regard, we also note that recently, in~\cite{Sutter20}, the authors introduced a quantum algorithm for computing the (discrete) conjugate of convex functions, which achieves a \emph{poly-logarithmic} time complexity in the number of data points.

One of the first and most widespread applications of these fast algorithms has been in solving the Hamilton-Jacobi equation \cite{Achdou14,Corrias96,Costes14}. 
Another interesting area of application is image processing, where the Legendre-Fenchel transform is commonly known as ``distance transform'' \cite{Felzen12,Lucet09}. 
Recently, in \cite{Jacobs19}, the authors used these algorithms to tackle the optimal transport problem with strictly convex costs, with applications in image processing and in numerical methods for solving partial differential equations. 
However, surprisingly, the application of these fast algorithms in solving discrete-time optimal control problems seems to remain largely unexplored. 
An exception is \cite{Caprio16}, where the authors use LLT to propose the ``fast value iteration'' algorithm for computing the fixed-point of the Bellman operator arising from a specific class of infinite-horizon, discrete-time DP problems. 
Indeed, the setup in \cite{Caprio16} corresponds to a subclass of problems considered in our study that allows for a ``perfect'' transformation of the minimization in the DP operation in the primal domain to an addition in the dual (conjugate) domain; this connection will be discussed in detail in Section~\ref{subsec:fvi}. 
Let us also note that the algorithms developed in~\cite{Felzen12,Lucet09} for distance transform can also potentially tackle the (discretized) optimal control problems similar to the ones considered in this study. 
In particular, these algorithms require the stage cost to be reformulated as a convex distance function of the current and next states. 
While this property might arise naturally, it can generally be restrictive as it is in our case. 

Another line of work, closely related to ours, involves algorithms that utilize max-plus algebra in solving, continuous-time, continuous-space, deterministic optimal control problems; see, e.g., \cite{Akian08,McEn03,McEn06}. 
These works exploit the compatibility of the Bellman operation with max-plus operations 
and approximate the value function as a max-plus linear combination. 
In particular, recently in \cite{Bach19,Bach20}, the authors used this idea to propose an approximate value iteration algorithm for deterministic MDPs with continuous state space. 
In this regard, we note that the proposed algorithms in the current study also implicitly involve representing cost functions as max-plus linear combinations. 
The key difference of the proposed algorithms is however to choose a dynamic, grid-like (factorized) set of slopes in the dual space to control the error and reduce the computational cost; we will discuss this point in more detail in Section~\ref{subsec:max-plus}

\noindent\textbf{Paper organization and summary of main results.} 
In this study, we consider the approximate implementation of the DP operation arising in the finite-horizon optimal control of discrete-time systems with continuous state and input spaces. 
The proposed approach involves discretization of the state space and is based on an alternative path that solves the dual problem corresponding to the DP operation by utilizing the LLT algorithm for discrete conjugation. 
After presenting some preliminaries in Section~\ref{sec:prel}, 
we provide the problem statement and its standard solution via the (discrete) DP algorithm (in the primal domain) in Section~\ref{sec:DP}. 
Sections~\ref{sec:conj DP} and \ref{sec:reduced complexity} contain our main results on the proposed alternative approach for solving the DP problem in the conjugate domain: 

\begin{itemize}

\item[(i)] \textbf{From minimization in primal domain to addition in dual domain:} In Section~\ref{sec:conj DP}, we introduce the discrete conjugate DP (d-CDP) algorithm (Algorithm~\ref{alg:d-CDP general}) for problems with deterministic \emph{input-affine} dynamics; see Figure~\ref{fig:CDP1} for the sketch of the algorithm.
In particular, we use the linearity of the dynamics in the input to effectively incorporate the operational duality of addition and infimal convolution, 
and transform the minimization in the DP operation to a simple addition at the expense of three conjugate transforms. 
This, in turn, leads to transferring the computational cost from the input domain~$\setc{U}$ to the dual state domain~$\setc{Y}$ (Theorem~\ref{thm:complexity d-CDP Alg 1}). 

\item[(ii)] \textbf{From quadratic to linear time complexity:} In Section~\ref{sec:reduced complexity}, we modify the proposed d-CDP algorithm (Algorithm~\ref{alg:d-CDP separ}) and reduce its time complexity (Theorem~\ref{thm:complexity d-CDP Alg 2})
for a subclass of problems with \emph{separable} data in the state and input variables; see Figure~\ref{fig:CDP2} for the sketch of the algorithm. 
In particular, for this class, the time complexity of computing the costs-to-go at each step is of $\ord(X+U)$, compared to the standard complexity of $\ord(XU)$.

\item[(iii)] \textbf{Error bounds and construction of dual domain:} We analyze the error of the proposed d-CDP algorithm and its modification (Theorems~\ref{thm:error d-CDP Alg 1} and \ref{thm:error d-CDP Alg 2}). 
The error analysis is based on two preliminary results on the error of discrete conjugation (Lemma~\ref{lem:conj vs. d-conj}) and approximate conjugation (Lemma~\ref{lem:conj via lerp} and Corollary~\ref{cor:disc conj via lerp}). 
Moreover, we use the results of our error analysis to provide concrete guidelines for the construction of a dynamic discrete dual space in the proposed algorithms (Remark~\ref{rem:constr Y}). 

\end{itemize}

In Section~\ref{sec:numerical ex}, we validate our theoretical results and compare the performance of the proposed algorithms with the benchmark d-DP algorithm through a synthetic numerical example. 
Further numerical examples (and descriptions of the extensions of the proposed algorithms) are provided in Appendix~\ref{app:num example}. 
Moreover, to facilitate the application of the proposed algorithms, we provide a MATLAB package:

\begin{itemize}

\item[(iv)] \textbf{The d-CDP MATLAB package:} The algorithms presented in this study and their extensions are available in the d-CDP MATLAB package~\cite{Kolari20}. 
A brief description of this package is provided in Appendix~\ref{app:matlab}. 
The numerical examples of this study are also included in the package and reproducible.  
\end{itemize}

Section~\ref{sec:remarks} concludes the paper by providing further remarks on the proposed algorithms such as their limitations and their relation to the existing schemes and algorithms in the literature. 

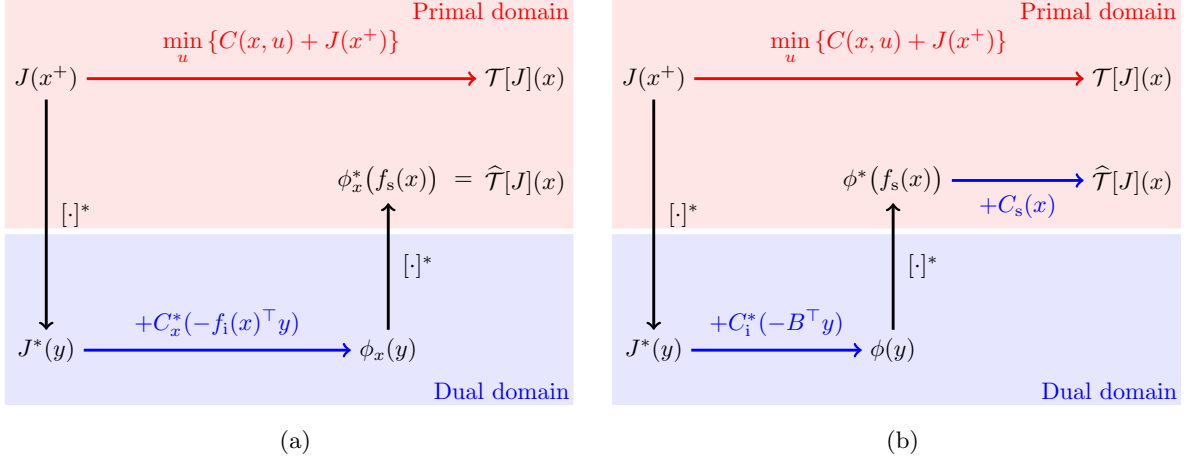
\begin{figure}[t]
\begin{subfigure}{.48\textwidth}
	\centering
	\scalebox{.9}{\begin{tikzpicture}

\filldraw[fill=red!10!white, draw=red!10!white] (-.6,1.17) rectangle (7.7,-2.2);
\filldraw[fill=blue!10!white, draw=blue!10!white] (-.6,-2.3) rectangle (7.7,-4.8);

\path (0,0) node(a)[align=left] {$J(x^+)$}
      (7,0) node(b)[align=right] {$\dpo[J](x)$}
      (0,-4) node(c)[align=left] {$\lftc{J}(y)$}
      (5,-4) node(d) {$\phi_x(y)$}
      (5,-1.5) node(e) {$\lftc{\phi_x} \big(\dynx(x)\big)$}
      (7,-1.5) node(f)[align=right] {$\cdpo[J](x)$}
      (6.66,-4.6) node(g)[align=left,blue] {Dual domain}
      (6.5,1) node(h)[align=left,red] {Primal domain}
      (6.1,-1.5) node [] {$=$};

\draw[->,very thick,red] (a) -- node[above=2pt] {$\min\limits_{u} \left\{ \cost(x,u) + J(x^+) \right\}$} (b);
\draw[->,very thick,black] (a) -- node[right=2pt,] {$\lftc{[\cdot]}$} (c);
\draw[->,very thick,blue] (c) -- node[above=2pt] {$+\lftc{\cost_x}(-\dynu(x)\tr y)$} (d);
\draw[->,very thick,black] (d) -- node[right=2pt] {$\lftc{[\cdot]}$} (e);
%\draw[<->,very thick,black,dashed] (b) -- node[left=2pt] {Thm.~\ref{thm:error d-CDP Alg 1}} (f);

\end{tikzpicture}}
	\caption{}
	\label{fig:CDP1}
\end{subfigure}%
\begin{subfigure}{.48\textwidth}
	\centering
	\scalebox{.9}{\begin{tikzpicture}

\filldraw[fill=red!10!white, draw=red!10!white] (-.6,1.17) rectangle (7.7,-2.2);
\filldraw[fill=blue!10!white, draw=blue!10!white] (-.6,-2.3) rectangle (7.7,-4.8);

\path (0,0) node(a)[align=left] {$J(x^+)$}
      (7,0) node(b)[align=right] {$\dpo[J](x)$}
      (0,-4) node(c)[align=left] {$\lftc{J}(y)$}
      (3.5,-4) node(d) {$\phi(y)$}
      (3.5,-1.5) node(e) {$\lftc{\phi} \big(\dynx(x)\big)$}
      (7,-1.5) node(f)[align=right] {$\cdpo[J](x)$}
	  (6.66,-4.6) node(g)[align=left,blue] {Dual domain}
      (6.5,1) node(h)[align=left,red] {Primal domain};

\draw[->,very thick,red] (a) -- node[above=2pt] {$\min\limits_{u} \left\{ \cost(x,u) + J(x^+) \right\}$} (b);
\draw[->,very thick,black] (a) -- node[right=2pt,] {$\lftc{[\cdot]}$} (c);
\draw[->,very thick,blue] (c) -- node[above=2pt] {$+\lftc{\costu}(-B\tr y)$} (d);
\draw[->,very thick,black] (d) -- node[right=2pt] {$\lftc{[\cdot]}$} (e);
\draw[->,very thick,blue] (e) --  node[below=2pt] {$+\costx(x)$} (f);
%\draw[<->,very thick,black,dashed] (b) -- node[left=2pt] {Thm.~\ref{thm:error d-CDP Alg 2}} (f);

\end{tikzpicture}}
	\caption{}
	\label{fig:CDP2}
\end{subfigure} 
\caption{Sketch of the proposed algorithms -- the standard DP operation in the primal domain (upper red paths) and the conjugate~DP (CDP) operation through the dual domain (bottom blue paths): (a)~Setting~\ref{Set:prob class I} with dynamics $x^+ = \dynx(x)+\dynu(x)\cdot u$ and generic cost $C(x,u)$;
(b) Setting~\ref{Set:prob class II} with dynamics $x^+ = \dynx(x)+B\cdot u$ and separable cost $C(x,u) = \costx(x)+\costu(u)$.
}
\label{fig:CDP}
\end{figure}

%===============================================================================
\section{Notations and preliminaries}
\label{sec:prel}
%===============================================================================
\subsection{General notations} \label{subsec:notations} 
 
We use $\R$ to denote the real line and $\Ru = \R \cup \{+\infty\}, \ \Rlu = \R \cup \{\pm\infty\}$ to denote its extensions. 
The standard inner product in $\R^n$ and the corresponding induced 2-norm are denoted by $\inner{\cdot}{\cdot}$ and $\norm{\cdot}$, respectively. 
We also use $\norm{\cdot}$ to denote the operator norm (w.r.t. the 2-norm) of a matrix; i.e., for $A \in \R^{m\times n}$, we denote $\norm{A} = \sup \{\norm{Ax}: \norm{x} = 1 \}$.
We use the common convention in optimization whereby the optimal value of an infeasible minimization (resp. maximization) problem is set to $+\infty$ (resp. $-\infty$). 

Continuous (infinite, uncountable) sets are denoted as $\setc{X}, \setc{Y}, \ldots$. 
For \emph{finite} (discrete) sets, we use the superscript~$\mathrm{d}$ as in $\setd{X}, \setd{Y}, \ldots$ to differentiate them form infinite sets. 
Moreover, we use the superscript~$\mathrm{g}$ to differentiate \emph{grid-like} (factorized) finite sets. 
Precisely, a grid~$\setg{X} \subset \R^n$ is the Cartesian product $\setg{X} = \Pi_{i=1}^{n} \setg{X}_{i} = \setg{X}_{1} \times \ldots \times \setg{X}_{n}$, where $\setg{X}_{i}$ is a finite set of real numbers $x_i^1 < x_i^2 < \ldots < x_i^{X_i}$. 
Assuming $X_i \geq 3$ for all $i = 1,\ldots, n$, we define $\setsg{X} \Let \Pi_{i=1}^{n} {\setsg{X}}_{i}$, 
where $ {\setsg{X}}_{i}= \setg{X}_i \setminus \{ x_i^1, x_i^{X_i} \}$; 
that is, $\setsg{X}$ is the \emph{sub-grid} derived by omitting the smallest and largest elements of $\setg{X}$ in each dimension.
The cardinality of a finite set $\setd{X}$ (or $\setg{X}$) is denoted by $X$. 
Let $\setc{X}, \setc{Y}$ be two arbitrary sets in $\R^n$.
The convex hull of $\setc{X}$ is denoted by $\co (\setc{X})$. 
The diameter of~$\setc{X}$ is defined as $\diam{\setc{X}} \Let \sup_{x,y \in \setc{X}} \norm{x-y}$.  
We use $\dist (\setc{X},\setc{Y}) \Let \inf_{x \in \setc{X}, y \in \setc{Y}} \norm{x-y}$ to denote the distance between $\setc{X}$ and $\setc{Y}$. 
The one-sided Hausdorff distance \emph{from}~$\setc{X}$ \emph{to}~$\setc{Y}$ is defined as $\dish(\setc{X}, \setc{Y}) \Let \sup_{x \in \setc{X}} \inf_{y \in \setc{Y}} \norm{x-y}$. 

For an extended real-valued function~$h:\R^n\ra \ol{\R}$, 
the effective domain of $h$ is defined by $\dom (h) \Let \{x \in \R^n: h(x) < +\infty \}$. 
The Lipschitz constant of~$h$ over a set~$\setc{X} \subset \dom (h)$ is denoted by 
$$\lip (h; \setc{X}) \Let  \sup_{x,y \in \setc{X}} \frac{|h(x)-h(y)|}{\norm{x-y}}.$$  
We also denote $\lip (h) \Let \lip \big(h; \dom (h)\big)$ and 
$\setc{L}(h) \Let \Pi_{i=1}^{n} \left[\lip_i^-(h), \lip_i^+(h)\right]$, where 
$\lip_i^+(h)$ (resp. $\lip_i^-(h)$) is the maximum (resp. minimum) slope of the function $h$ along the $i$-th dimension, i.e.,
\begin{align*}
\lip_i^+(h) &\Let \sup \left\{ \frac{h(x) - h(y)}{x_i - y_i}: x,y \in \dom (h), \ x_i > y_i, \ x_j = y_j\ (j\neq i) \right\}, \\
\lip_i^-(h) &\Let \inf \left\{ \frac{h(x) - h(y)}{x_i - y_i}: x,y \in \dom (h), \ x_i > y_i, \ x_j = y_j\ (j\neq i) \right\}.
\end{align*}
The subdifferential of~$h$ at a point~$x \in \R^n$ is defined as 
$$
\partial h(x) \Let \big\{y \in \R^n: h(\tilde{x}) \geq h(x) + \inner{y}{\tilde{x}-x}, \forall \tilde{x} \in \dom (h) \big\}.$$ 

We report the complexities using the standard big O notations $\ord$ and $\wt{\ord}$, where the latter hides the logarithmic factors. 
In this study, we are mainly concerned with the dependence of the computational complexities on \emph{the size of the finite sets} involved (discretization of the primal and dual domains). 
In particular, we ignore the possible dependence of the computational complexities on the dimension of the variables, unless they appear in the power of the size of those discrete sets; e.g., the complexity of a single evaluation of an analytically available function is taken to be of $\ord(1)$, regardless of the dimension of its input and output arguments.  
For the reader's convenience, we also provide the list of the most important objects used throughout this article in Table~\ref{tab:notation}. 
\begin{table}[h]
\caption{List of the most important notational conventions.}
\label{tab:notation}
%\vskip -0.7cm
\begin{center}
\begin{small}
%\begin{sc}
\begin{tabular}{rlc}
\toprule
\multicolumn{2}{l}{Notation \& Description} & Definition \\
\midrule
LERP & Multilinear interpolation \& extrapolation  & --  \\
LLT & Linear-time Legendre Transform  & --  \\ 
$\disc{h}$ & Discretization of the function $h$  & --  \\
$\lerp{\disc{h}}$ & Extension of the discrete function $\disc{h}$  & --  \\
$\llerp{\disc{h}}$ & LERP extension of the discrete function $\disc{h}$ (with grid-like domain)  & --  \\
$\lftc{h}$ & Conjugate of $h$  & \eqref{eq:conj}  \\
$\lftd{h}$ & Discrete conjugate of $h$ (conjugate of $\disc{h}$)  & \eqref{eq:conj disc}  \\
$\bcc{h}$ & Biconjugate of $h$  & \eqref{eq:biconj}  \\
$\bcdd{h}$ & Discrete biconjugate of $h$  & \eqref{eq:biconj doub disc}  \\
$\dpo$ & Dynamic Programming (DP) operator & \eqref{eq:DP op} \& \eqref{eq:DP op separ} \\
$\ddpo$ & Discrete DP (d-DP) operator & \eqref{eq:d-DP op}  \\
$\cdpo$ & Conjugate DP (CDP) operator & \eqref{eq:CDP op conj}  \\
$\dcdpo$ & Discrete CDP (d-CDP) operator & \eqref{eq:d-CDP op} \& \eqref{eq:d-CDP op separ}  \\
$\mdcdpo$ & Modified d-CDP operator  & \eqref{eq:d-CDP op separ modified}  \\
 \bottomrule
\end{tabular}
%\end{sc}
\end{small}
\end{center}
%\vskip -0.1in
\end{table}

\subsection{Extension of discrete functions} \label{subsec:extension operation}

Consider an extended real-valued function~$h:\R^n\ra \ol{\R}$, and its discretization $\disc{h}: \setd{X} \ra \Ru$, where $\setd{X}$ is a finite subset of $\R^n$. 
We use the superscript $\mathrm{d}$, as in $\disc{h}$, to denote the discretization of $h$. 
We particularly use this notation in combination with a second operation to emphasize that the second operation is applied on the discretized version of the operand. 
In particular, we use $\lerp{\disc{h}}:\R^n \ra \Ru$ to denote the extension of the discrete function~$\disc{h}:\setd{X} \ra \Ru$. 
The extension can be considered as a generic parametric approximation $\lerp{\disc{h}}[\theta]:\R^n \ra \Ru$, where the parameters $\theta$ are computed using regression, i.e., by fitting $\lerp{\disc{h}}[\theta]$ to the data points $\disc{h}:\setd{X} \ra \Ru$.

\begin{Rem}[Complexity of extension operation] \label{rem:extension operation}
We use $E$ to denote the complexity of a generic extension operator. 
That is, for each $x\in\R^n$, the time complexity of the single evaluation $\lerp{\disc{h}}(x)$ is assumed to be of $\ord(E)$, with $E$ (possibly) being a function of $X$.  
\end{Rem} 

For example, for the linear approximation $\lerp{\disc{h}}(x) = \sum_{i=1}^{B} \theta_i \cdot b_i(x)$, we have $E = B$ (the size of the basis), while for the kernel-based approximation $\lerp{\disc{h}} (x)= \sum_{\bar{x} \in \setd{X}}\theta_{\bar{x}} \cdot r(x,\bar{x})$, we generally have $E \leq X$. 
A kernel-based approximator of interest in the following sections is the \emph{multilinear interpolation \& extrapolation} (LERP) of a discrete function with a \emph{grid-like} domain; see \cite[App.~D]{Kirk2010} for a description of LERP in the two-dimensional case.. 
Hence, we denote this operation with the different notation $\llerp{\disc{h}}: \R^n \ra \Ru$ for the discrete function $\disc{h}:\setg{X} \ra \Ru$. 
Notice that the LERP extension preserves the value of the function at the discrete points, i.e, $\llerp{\disc{h}}(x) = \disc{h} (x)$ for all $x \in \setg{X}$. 
In order to facilitate our complexity analysis in subsequent sections, we discusses the computational complexity of LERP in the following remark. 

\begin{Rem}[Complexity of LERP] \label{rem:complexity of LERP}
Given a discrete function~$\disc{h}: \setg{X} \ra \R$ with a grid-like domain~$\setg{X} \subset \R^n$,  
the time complexity of a single evaluation of the LERP extension $\llerp{\disc{h}}$ at a point~$x \in \R^n$  is of $\ord(2^n+\log X) = \wt{\ord}(1)$ if $\setg{X}$ is non-uniform, and of $\ord(2^n) = \ord(1)$ if $\setg{X}$ is uniform. 
To see this, note that, in the case $\setg{X}$ is non-uniform, LERP requires $\ord(\log X)$ operations to find the position of $x$ w.r.t. the grid points, using binary search. 
If $\setg{X}$ is a uniform grid, this can be done in  $\ord(n)$ time. 
Upon finding the position of $x$, LERP then involves a series of one-dimensional linear interpolations or extrapolations along each dimension, which takes $\ord(2^n)$ operations.    
\end{Rem} 

For a convex function~$h: \R^n \ra \Ru$, we have $\partial h (x) \neq \emptyset$ for all~$x$ in the relative interior of~$\setc{X}$ \cite[Prop.~5.4.1]{Bertsekas09}. 
This characterization of convexity can be extended to discrete functions. 
A discrete function~$\disc{h}: \setd{X} \ra \R$ is called \emph{convex-extensible} if $\partial \disc{h} (x) \neq \emptyset$ for all~$x \in \dom(h) = \setd{X}$. 
Equivalently, $\disc{h}$ is convex-extensible, if it can be extended to a convex function $\lerp{\disc{h}}:\R^n \ra \Ru$ such that $\lerp{\disc{h}}(x) = \disc{h} (x)$ for all~$x \in \setd{X}$; we refer the reader to, e.g., \cite{Murota03} for different extensions of the notion of convexity to discrete functions. 

\subsection{Legendre-Fenchel Transform}

Consider an extended-real-valued function $h: \R^n \ra \Ru$, with a nonempty effective domain $\dom(h) = \setc{X}$. 
The Legendre-Fenchel transform (convex conjugate) of~$h$ is the function
\begin{equation} \label{eq:conj}
\lftc{h}: \R^n \ra \Rlu: y \mapsto \sup_{x \in \setc{X}} \left\{ \inner{y}{x} - h(x) \right\}.
\end{equation} 
Note that the conjugate function $\lftc{h}$ is convex by construction. 
In this study, we particularly consider \emph{discrete} conjugation, which involves computing the conjugate function using the discretized version $\disc{h}: \setd{X} \ra \Ru$ of the function $h$, where $\setd{X} \cap \setc{X} \neq \emptyset$. 
We use the notation $\lftd{[\cdot]}$, as opposed the standard notation $\lftc{[\cdot]}$, for discrete conjugation; that is, 
\begin{equation} \label{eq:conj disc}
\lftd{h} = \lftc{[\disc{h}]}: \R^n \ra \R: y \mapsto \max_{x \in \setd{X}} \left\{ \inner{y}{x} - \disc{h}(x) \right\}.
\end{equation} 
The biconjugate of $h$ is the function 
\begin{equation} \label{eq:biconj}
\bcc{h} = \lftc{[\lftc{h}]}: \R^n \ra \Rlu: x \mapsto \sup_{y \in \R^n} \{\inner{x}{y} - \lftc{h}(y) \} = \sup_{y \in \R^n} \ \inf_{z \in \setc{X}} \left\{ \inner{x-z}{y} + h(z) \right\}.
\end{equation}
Using the notion of discrete conjugation $\lftd{[\cdot]}$, we also define the (doubly)  discrete biconjugate
\begin{equation} \label{eq:biconj doub disc}
\bcdd{h} = \lftd{[\lftd{h}]}: \R^n \ra \R: x \mapsto \max_{y \in \setd{Y}} \ \{\inner{x}{y} - \lftdd{h}(y) \} =  \max_{y \in \setd{Y}} \ \min_{z \in \setd{X}} \left\{ \inner{x-z}{y} + \disc{h}(z) \right\},
\end{equation}
where $\setd{X}$ and $\setd{Y}$ are finite subsets of $\R^n$ such that $\setd{X} \cap \setc{X} \neq \emptyset$.

The Linear-time Legendre Transform (LLT) is an efficient algorithm for computing the discrete conjugate over a finite \emph{grid-like} dual domain. 
Precisely, to compute the conjugate of the function~$h: \setc{X} \ra \R$, LLT takes its discretization $\disc{h}: \setd{X} \ra \R$ as an input, and outputs $\lftdd{h}: \setg{Y} \ra \R$, for the grid-like dual domain $\setg{Y}$. 
That is, LLT is equivalent to the operation $\lftdd{[\cdot]}$. 
We refer the interested reader to \cite{Lucet97} for a detailed description of the LLT algorithm. 
We will use the following result for analyzing the computational complexity of the proposed algorithms.

\begin{Rem}[Complexity of LLT] \label{rem:complexity of LLT}
Consider a function $h: \R^n \ra \Ru$ and its discretization over a grid-like set $\setg{X} \subset \R^n$ such that $\setg{X} \cap \dom(h) \neq \emptyset$. 
LLT computes the discrete conjugate function $\lftdd{h}:\setg{Y} \ra \R$ using the data points $\disc{h}: \setg{X} \ra \Ru$, with a time complexity of $\ord\big(\Pi_{i=1}^n(X_i + Y_i)\big)$, where $X_i \ (\text{resp. }Y_i)$ is the cardinality of the $i$-th dimension of the grid $\setg{X} \ (\text{resp. }\setg{Y})$. 
In particular, if the grids $\setg{X}$ and $\setg{Y}$ have approximately the same cardinality in each dimension, then the time complexity of LLT is of $\ord(X+Y)$ \cite[Cor.~5]{Lucet97}. 
\end{Rem}

Hereafter, to simplify the exposition, we consider the following assumption. 

\begin{As}[Grid sizes in LLT]\label{As:grid size LLT}
The primal and dual grids used for LLT operation have approximately the same cardinality in each dimension.
\end{As}

\subsection{Preliminary results on conjugate transform}\label{subsec:pre lemmas}

In what follows, we provide two preliminary lemmas on the error of discrete conjugate transform and its approximate version. 
Although tailored for the error analysis of the proposed algorithms, we present these results in a  generic format to facilitate their possible application/extension beyond this study. 

Let us begin with recalling some of the notations introduced so far. 
Consider a function $h:\R^n \ra \Ru$ with a nonempty effective domain $\setc{X} = \dom(h)$, and 
its discretization $\disc{h}:\setd{X} \ra \R$ where $\setd{X} \subset \setc{X}$. 
Let $\lftc{h}: \R^n \ra \Ru$ be the conjugate~\eqref{eq:conj} of $h$,  
and also let $\lftd{h}: \R^n \ra \R$ be the discrete conjugate~\eqref{eq:conj disc} of $h$, using the primal discrete domain $\setd{X}$. 

\begin{Lem}[Conjugate vs. discrete conjugate]\label{lem:conj vs. d-conj}
Let $h$ be proper, closed, and convex. Then, 
\begin{equation}\label{eq:conj vs. d-conj I}
0 \leq \lftc{h}(y) - \lftd{h}(y) \leq \min\limits_{x \in \partial  \lftc{h}(y)} \bigg\{ \big[ \norm{y} + \lip \big( h; \{x\} \cup \setd{X} \big) \big] \cdot \dist(x,\setd{X}) \bigg\} \teL \wt{e}_1(y, h, \setd{X}), \quad \forall y \in \R^n. 
\end{equation}
If, moreover, $\setc{X}$ is compact and $h$ is Lipschitz continuous, then 
\begin{equation}\label{eq:conj vs. d-conj II}
0 \leq \lftc{h}(y) - \lftd{h}(y) \leq \big[ \norm{y} + \lip (h) \big] \cdot \dish (\setc{X},\setd{X})\teL \wt{e}_2(y,h,\setd{X}), \quad \forall y \in \R^n. 
\end{equation}
\end{Lem} 

\begin{proof}
See Appendix~\ref{proof:conj vs. d-conj}.
\end{proof}

The preceding lemma indicates that discrete conjugation leads to an under-approximation of the conjugate function, 
with the error depending on the discrete representation~$\setd{X}$ of the primal domain~$\setc{X}$. 
In particular, the inequality~\eqref{eq:conj vs. d-conj I} implies that for $y \in \R^n$, if $\setd{X}$ contains $x \in \partial  \lftc{h}(y)$, which is equivalent to $y \in \partial h(x)$ by the assumptions, then $\lftd{h}(y) = \lftc{h}(y)$. 

We next present another preliminary however vital result on approximate conjugation. 
Let $\lftcd{h}: \setg{Y} \ra \R$ be the discretization of $\lftc{h}$ over the grid-like dual domain $\setg{Y} \subset \dom(\lftc{h}) \subseteq \R^n$. 
Also, let $\llerp{\lftcd{h}}: \R^n \ra \R$ be the extension of $\lftcd{h}$ using LERP. 
The approximate conjugation is then simply the approximation of $\lftc{h} (y)$ via $\llerp{\lftcd{h}}(y)$ for $y \in \R^n$. 
This approximation introduces a one-sided error: 

\begin{Lem}[Approximate conjugation using LERP] \label{lem:conj via lerp}
Let $\setc{X} = \dom(h)$ be compact. Then, 
\begin{equation} \label{eq:conj via lerp error}
0 \leq \llerp{\lftcd{h}} (y)  - \lftc{h} (y) \leq \diam{\setc{X}} \cdot \dist(y,\setg{Y}), \quad \forall y \in \co (\setg{Y}).
\end{equation}
If, moreover, the dual grid $\setg{Y}$ is such that $\co(\setsg{Y}) \supseteq \setc{L}(h)$, then 
\begin{equation} \label{eq:conj via lerp error 2}
0 \leq \llerp{\lftcd{h}} (y)  - \lftc{h} (y) \leq \diam{\setc{X}} \cdot \dish\big(\co(\setg{Y}),\setg{Y}\big), \quad \forall y \in \R^n.
\end{equation}
\end{Lem}

\begin{proof}
See Appendix~\ref{proof:conj via lerp}.
\end{proof}

As expected, the error due to the discretization~$\setg{Y}$ of the dual domain~$\setc{Y}$ depends on the resolution of the discrete dual domain. 
We also note that the condition $\co(\setsg{Y}) \supseteq \setc{L}(h)$ in the second part of the preceding lemma (which implies that $h$ is Lipschitz continuous), essentially requires the dual grid $\setg{Y}$ to \emph{more than cover the range of slopes} of the function $h$. 

The algorithms developed in this study use LLT to compute discrete conjugate functions.  
However, as we will see, we sometimes require the value of the conjugate function at points other than the dual grid points used in LLT. 
To solve this issue, we use the same approximation described above, but now for discrete conjugation. 
In this regard, we note that the result of Lemme~\ref{lem:conj via lerp} also holds for discrete conjugation. 
To be precise, consider the discrete function $\disc{h}:\setd{X}\ra \R$. 
Let $\lftdd{h}: \setg{Y} \ra \R$ be the discretization of $\lftd{h}$ over the grid-like dual domain $\setg{Y} \subset \R^n$, 
and let $\llerp{\lftdd{h}}: \R^n \ra \R$ be the extension of $\lftdd{h}$ using LERP. 

\begin{Cor}[Approximate discrete conjugation using LERP] \label{cor:disc conj via lerp}
We have
\begin{equation} \label{eq:disc conj via lerp error}
0 \leq \llerp{\lftdd{h}} (y)  - \lftd{h} (y) \leq \diam{\setd{X}} \cdot \dist(y,\setg{Y}), \quad \forall y \in \co (\setg{Y}). 
\end{equation}
If, moreover, the dual grid $\setg{Y}$ is such that $\co(\setsg{Y}) \supseteq \setc{L}(\disc{h})$, 
then 
\begin{equation} \label{eq:disc conj via lerp error 2}
0 \leq \llerp{\lftdd{h}} (y)  - \lftd{h} (y) \leq \diam{\setd{X}} \cdot \dish\big(\co(\setg{Y}),\setg{Y}\big), \quad \forall y \in \R^n.
\end{equation}
\end{Cor}

\begin{proof}
See Appendix~\ref{proof:disc conj via lerp}.
\end{proof}

%===============================================================================
\section{Problem statement and standard solution}
\label{sec:DP}
%===============================================================================

In this study, we consider the optimal control of discrete-time systems
\begin{equation} \label{eq:dyn}
x_{t+1} = \dyn(x_t,u_t) ,\quad t = 0, \ldots, T-1,
\end{equation}
where $\dyn : \R^n \times \R^m \ra \R^n$  describes the dynamics, and $T \in \N$ is the finite horizon. 
Here, we focus on deterministic dynamics. 
However, we note that the proposed algorithms in the subsequent sections can be extended to handle  stochastic dynamics with additive noise; see Appendix~\ref{subsec:extension stochastic} for more details. 
We also consider state and input constraints of the form
\begin{eqnarray} \label{eq:const}
\left\{
\begin{array}{lcl}
x_t \in \setc{X} \subset \R^n \  &\text{for}& \  t \in \{0, \ldots, T\}, \\
u_t \in \setc{U} \subset \R^m \  &\text{for}& \  t \in \{0, \ldots, T-1\}. \\
\end{array}
\right.
\end{eqnarray}
Let $\cost: \setc{X} \times \setc{U} \ra \Ru$ and $\cost_T:\setc{X} \ra \R$ be the stage and terminal cost functions, respectively. 
In particular, notice that we let the stage cost $\cost$ take $+\infty$ for $(x,u) \in \setc{X} \times \setc{U}$ so that it can embed the \emph{state-dependent input constraints}.
For an initial state~$x_0 \in \setc{X}$, the cost incurred by the state trajectory $\mathbf{x} = (x_0, \ldots,x_T)$ in response to the input sequence $\mathbf{u} = (u_0, \ldots,u_{T-1})$ is given by 
\begin{equation*}
J(x_0, \mathbf{u}) = \ssum_{t=0}^{T-1} \cost (x_t,u_t) + \cost_T(x_T).
\end{equation*}
The problem of interest is then to find an optimal control sequence~$\mathbf{u}\opt (x_0)$, that is, a solution to the minimization problem 
\begin{equation}\label{eq:opt cont prob}
J\opt (x_0) = \min_{\mathbf{u}} \left\{ J(x_0, \mathbf{u}) : \eqref{eq:dyn} \ \& \ \eqref{eq:const}   \right\}.
\end{equation} 
Throughout this study, we assume that the problem data satisfy the following conditions.

\begin{As}[Problem data]\label{As:general} 
The dynamics, constraints, and costs have the following properties:
\begin{enumerate}[label=(\roman*)]

\item \label{As:dyn} \textbf{\emph{Dynamics.}} The dynamics~$\dyn : \R^n \times \R^m \ra \R^n$ is locally Lipschitz continuous. 

\item \label{As:constr} \textbf{\emph{Constraints.}} The constraint sets $\setc{X}$ and $\setc{U}$ are compact. 
Moreover, the set of admissible inputs 
$
\setc{U}(x) \Let \{ u \in \setc{U}:  \cost(x,u) < +\infty, \ \dyn(x,u) \in \setc{X}\}
$ 
is nonempty for all $x \in \setc{X}$.

\item \label{As:cost} \textbf{\emph{Cost functions.}} The stage cost $\cost: \setc{X} \times \setc{U} \ra \Ru$ has a compact effective domain. Moreover, $\cost$ and $\cost_T$ are Lipschitz continuous.

\end{enumerate}
\end{As} 

The properties laid out in Assumption~\ref{As:general} imply that the set $\setc{U}(x)$ of admissible inputs is nonempty and compact, and the objective in~\eqref{eq:opt cont prob} is continuous 
(compactness of $\setc{U}(x)$ follows from compactness of $\dom (\cost)$ and $\setc{X}$, and continuity of $\dyn$). 
Hence, the optimal value in~\eqref{eq:opt cont prob} is achieved. 

To solve the problem described above using DP, we have to solve the Bellman equation
\begin{equation*}\label{eq:DP recursion}
J_{t}(x_t) = \min_{u} \left\{ \cost (x_t,u_t) + J_{t+1}(x_{t+1}):  \eqref{eq:dyn} \ \& \ \eqref{eq:const} \right\}, \quad x_t \in \setc{X},
\end{equation*}
backward in time $t = T-1, \ldots,0$, initialized by $J_T = \cost_T$. 
The iteration finally outputs $J_0 = J\opt$ \cite[Prop.~1.3.1]{Bertsekas05}. 
To simplify the exposition, let us embed the state and input constraints in the cost functions ($\cost$ and $J_t$) by extending them to infinity outside their effective domain. 
Let us also drop the time subscript $t$ and focus on a single step of the recursion by defining the DP operator
\begin{equation} \label{eq:DP op}
\dpo [J](x) \Let \min_{u} \left\{ \cost(x,u) + J\big(\dyn(x,u)\big) \right\}, \quad x \in \setc{X},
\end{equation}
so that $J_t = \dpo [J_{t+1}] = \dpo^{(T-t)} [J_T]$ for $t = T-1, \ldots,0$. 
Notice that the DP operation~\eqref{eq:DP op} requires solving an infinite number of optimization problems for the continuous state space $\setc{X}$. 
Except for a few cases with an available closed-form solution, the exact implementation of DP operation is impossible. 
A standard approximation scheme is then to incorporate function approximation techniques and solve~\eqref{eq:DP op} for a finite sample  (i.e., a discretization) of the underlying continuous state space. 
Precisely, we consider solving the optimization in~\eqref{eq:DP op} for a finite number of $x \in \setg{X}$, 
where $\setg{X} \subset \setc{X}$ is a grid-like discretization of the state space, to derive the output $\disc{[\dpo [J]]}:\setg{X}\ra \R$. 
This also means that the DP operator $\dpo$ now takes the discrete function $\disc{J}:\setg{X} \ra \R$ (the output of the previous iteration) as the input. 
Hence, along with the discretization of the state space, we also need to consider some form of function approximation for the cost-to-go function, that is, an extension $\lerp{\disc{J}}: \setc{X} \ra \R$ of the function~$\disc{J}: \setg{X} \ra \R$. 
What remains to be addressed is the issue of solving the minimization  
\begin{equation*}
\min_{u \in \setc{U}} \left\{\cost(x,u) + \lerp{\disc{J}}\big(\dyn(x,u)\big) \right\},
\end{equation*}
for each $x \in \setg{X}$, where the next step cost-to-go is approximated by the extension $\lerp{\disc{J}}$. 
This minimization problem is often a difficult, non-convex problem. 
Again, a common approximation involves enumeration over a proper discretization $\setd{U} \subset \setc{U}$ of the inputs space. 
We assume that the joint discretization of the state-input space is ``proper'' in the sense that the feasibility condition of Assumption~\ref{As:general}-\ref{As:constr} holds for the discrete state-input space:

\begin{As}[Feasible discretization]\label{As:feasible discrete}
The discrete state space~$\setg{X}\subset\setc{X}$ and input space~$\setd{U}\subset\setc{U}$ are such that $\setd{U}(x) \Let \setc{U}(x)\cap \setd{U}$ is nonempty for all $x \in \setg{X}$.
\end{As}

These approximations introduce some error which, under some regularity assumptions, depends on the discretization of the state and input spaces and the extension operation; see Proposition~\ref{prop:error d-DP}. 
Incorporating these approximations, we can introduce the \emph{discrete} DP (d-DP) operator as follows
\begin{equation} \label{eq:d-DP op}
\ddpo [\disc{J}](x) \Let \min_{u \in \setd{U}} \left\{ \cost(x,u) + \lerp{\disc{J}}\big(\dyn(x,u)\big) \right\}, \quad x \in \setg{X}.
\end{equation}
The d-DP operator/algorithm will be our benchmark for evaluating the performance of the alternative algorithms developed in this study. 
To this end, we discuss the time complexity of the d-DP operation in the following remark.  

\begin{Rem}[Complexity of d-DP] \label{rem:complexity d-DP}
Let the time complexity of a single evaluation of the extension operator~$\lerp{[\cdot]}$ in \eqref{eq:d-DP op} be of $\ord (E)$. 
Then, the time complexity of the d-DP operation~\eqref{eq:d-DP op} is of $\ord\big(XUE\big)$. 
\end{Rem} 

Let us clarify that the scheme described above essentially involves approximating a continuous-state/action MDP with a finite-state/action MDP, and then applying the (fitted) value iteration algorithm. 
In this regard, we note that $\ord(XU)$ is the best-existing time-complexity in the literature for finite MDPs; see, e.g., \cite{Bach19, Sidford18}. 
Indeed, regardless of the problem data, the d-DP algorithm involves solving a minimization problem for each $x \in \setg{X}$, via enumeration over $u \in \setd{U}$. 
However, as we will see in the subsequent sections, for certain classes of optimal control problems, it is possible to exploit the structure of the underlying continuous setup to avoid the minimization over the input and achieve a lower time complexity.

%===============================================================================
\section{From minimization in primal domain to addition in dual domain}
\label{sec:conj DP}
%===============================================================================

We now introduce a general class of problems that allows us to employ conjugate duality for the DP problem 
and hence propose an alternative path for implementing the corresponding operator. 
In particular, we show that the linearity of dynamics in the input is the key property in developing the alternative solution, 
whereby the minimization in the primal domain is transformed to an addition in the dual domain at the expense of three conjugate transforms. 
The problem class of interest is as follows:

\begin{Set}\label{Set:prob class I} 
The dynamics is input-affine, that is, $\dyn(x,u) = \dynx(x) + \dynu(x) \cdot u$, where $\dynx : \R^n \ra \R^n$ is the ``state'' dynamics, and $\dynu : \R^n \ra \R^{n\times m}$ is the ``input'' dynamics.
\end{Set}

\subsection{The d-CDP algorithm} \label{subsec:d-CPP alg 1}

Alternatively, we can approach the optimization problem in the DP operation~\eqref{eq:DP op} in the dual domain. 
To this end, let us fix $x \in \setc{X}$, and consider the following reformulation of the problem~\eqref{eq:DP op}:
\begin{equation*}
\dpo [J](x) =  \min_{u,z} \left\{ \cost(x,u) + J(z): z = \dyn(x,u)  \right\}.
\end{equation*}
Notice how for input-affine dynamics of Setting~\ref{Set:prob class I}, this formulation resembles the infimal convolution \eqref{eq:inf conv} (by taking $w_1 = z$ and $w_2 = u$, the equality constraint becomes $w_1 - \dynu(x)\cdot w_2 = \dynx(x)$). In this regard, consider the corresponding dual problem
%\begin{subequations} 
\begin{align}\label{eq:CDP op}
\cdpo [J](x) \Let \max_{y} \ \min_{u , z} \ \{\cost(x,u) + J(z)+ \inner{y}{\dyn(x,u)-z}\}, 
\end{align}
%\end{subequations}
where~$y \in \R^n$ is the dual variable. 
Indeed, for input-affine dynamics, we can derive an equivalent formulation for the dual problem~\eqref{eq:CDP op}, which forms the basis for the proposed algorithms. 

\begin{Lem}[CDP operator] \label{lem:CDP op}
Let 
\begin{equation} \label{eq:conj cost func alg 1}
\lftc{\cost_x}(v) \Let \max_{u} \big\{ \inner{v}{u} - \cost(x,u) \big\}, \quad v \in \R^m,
\end{equation}
denote the partial conjugate of the stage cost w.r.t. the input variable~$u$.
Then, for the input-affine dynamics of Setting~\ref{Set:prob class I}, the operator~$\cdpo$~\eqref{eq:CDP op} equivalently reads as
\begin{subequations} \label{eq:CDP op conj}
\begin{align}
& \phi_x(y) \Let \lftc{\cost_x}(-\dynu(x)\tr y) + \lftc{J}(y), & y \in \R^n, \label{eq:CDP op conj b} \\
& \cdpo [J](x) = \lftc{\phi_x} \big(\dynx(x)\big), &  x \in \setc{X}. \label{eq:CDP op conj a}
\end{align}
\end{subequations}
\end{Lem}

\begin{proof}
See Appendix~\ref{proof:CDP op}.
\end{proof}

As we mentioned, the construction above suggests an alternative path for computing the output of the DP~operator through the conjugate domain. 
We call this alternative approach \emph{conjugate} DP (CDP). 
Figure~\ref{fig:CDP1} characterizes this alternative path schematically. 
Numerical implementation of CDP operation requires the computation of conjugate functions. 
In particular, as shown in Figure~\ref{fig:CDP1}, CDP operation involves three conjugate transforms. 
For now, we assume that the partial conjugate $\lftc{\cost_x}$ of the stage cost in~\eqref{eq:conj cost func alg 1} is analytically available. 
We note however that one can also consider a numerical scheme to approximate this
conjugation; see Appendix~\ref{subsec:extension num conj} for further details.

\begin{As}[Conjugate of stage cost]\label{As:conj cost func alg 1}
The conjugate function~$\lftc{\cost_x}$~\eqref{eq:conj cost func alg 1} is analytically available. 
That is, the time complexity of evaluating $\lftc{\cost_x} (v)$ for each $v \in \R^m$ is of $\ord (1)$.
\end{As}

The two remaining conjugate operations of the CDP path in Figure~\ref{fig:CDP1} are handled numerically. 
In particular, we again take a sample-based approach and compute $\cdpo [J]$ for a finite number of states $x \in \setg{X}$. 
To be precise, for a grid-like discretization~$\setg{Y}$ of the dual domain, 
we employ LLT to compute $\lftdd{J}: \setg{Y} \ra \R$ using the data points $\disc{J}: \setg{X} \ra \R$. 
Proper construction of $\setg{Y}$ will be discussed shortly.
Now, let
$$
\disc{\psi}_x(y) \Let \lftc{\cost_x}(-\dynu(x)\tr y) + \lftdd{J}(y), \quad y \in \setg{Y},
$$
be a discrete \emph{approximation} of $\phi_x$ in~\eqref{eq:CDP op conj b}. 
The approximation stems from the fact that we used the discrete conjugate~$\lftd{J}$ instead of the conjugate~$\lftc{J}$. 
Using this object, we can also handle the last conjugate transform in Figure~\ref{fig:CDP1} numerically, and approximate $\lftc{\phi_x}\big(\dynx(x)\big)$ in~\eqref{eq:CDP op conj a} by
$$
\lftd{\psi_x} \big(\dynx(x)\big) = \max_{y \in \setg{Y}} \ \{ \inner{\dynx(x)}{y} - \disc{\psi_x}(y) \},
$$
via enumeration over~$y \in \setg{Y}$. 
Based on the construction described above, we can introduce the \emph{discrete} CDP (d-CDP) operator as follows
\begin{subequations} \label{eq:d-CDP op}
\begin{align}
&\lftdd{J} (y) = \max_{x \in \setg{X}}\left\{\inner{y}{x} - \disc{J}(x) \right\}, & y \in \setg{Y}, \\
&\disc{\psi}_x(y) = \lftc{\cost_x}(-\dynu(x)\tr y) + \lftdd{J}(y), & y \in \setg{Y}, \label{eq:d-CDP op b} \\
& \dcdpo [\disc{J}](x) \Let \lftd{\psi_x} \big(\dynx(x)\big), &  x \in \setg{X}. \label{eq:d-CDP op a}
\end{align}
\end{subequations}
Algorithm~\ref{alg:d-CDP general} provides the pseudo-code for the numerical implementation of the d-CDP operation~\eqref{eq:d-CDP op}. 
In the next subsection, we analyze the complexity and error of Algorithm~\ref{alg:d-CDP general}. 

\begin{algorithm}[t]
\begin{small}
   \caption{Implementation of the d-CDP operator~\eqref{eq:d-CDP op} for Setting~\ref{Set:prob class I}.} 
   \label{alg:d-CDP general}
\begin{algorithmic}[1]
	\REQUIRE dynamics~$\dynx: \R^n \ra \R^n, \ \dynu : \R^n \ra \R^{n\times m}$; \\
	discrete cost-to-go~(at $t+1$) $\disc{J}: \setg{X} \ra \R$;	\\
	conjugate of stage cost~$\lftc{\cost_x}: \R^m \ra \R$ for $x \in \setg{X}$;  \\
	grid $\setg{Y} \subset \R^n$; 
	\ENSURE discrete cost-to-go (at $t$)~$\dcdpo [\disc{J}](x): \setg{X} \ra \R$.
  	
	\vspace{.2cm}
  	
    \STATE use LLT to compute $\lftdd{J}:\setg{Y} \ra \R$ from $\disc{J}: \setg{X} \ra \R$; \label{line_alg_g:LLT of J}
    
    \vspace{.1cm}
    
    \FOR{each $x \in \setg{X}$} \label{line_alg_g:for loop over x}

    	\STATE  $\disc{\psi}_x(y) \gets \lftc{\cost_x}(-\dynu(x)\tr y) + \lftdd{J}(y)$ for $y \in \setg{Y}$; \label{line_alg_g:h}

  		\STATE $\dcdpo [\disc{J}](x) \gets \max\limits_{y \in \setg{Y}}\{\inner{\dynx(x)}{y} - \disc{\psi}_x(y) \}$. \label{line_alg_g:output}

  	\ENDFOR

\end{algorithmic}
\end{small}
\end{algorithm}

\subsection{Analysis of d-CDP algorithm}\label{subsec:analys alg 1} 

We begin with the computational complexity of Algorithm~\ref{alg:d-CDP general}. 

\begin{Thm}[Complexity of d-CDP Algorithm~\ref{alg:d-CDP general}] \label{thm:complexity d-CDP Alg 1}
Let Assumptions~\ref{As:grid size LLT} and~\ref{As:conj cost func alg 1} hold. 
Then, the implementation of the d-CDP operator~\eqref{eq:d-CDP op} via Algorithm~\ref{alg:d-CDP general} requires $\ord (XY)$ operations.  
\end{Thm}

\begin{proof}
See Appendix~\ref{proof:complexity d-CDP Alg 1}.
\end{proof}

Recall that the time complexity of the d-DP operator~\eqref{eq:d-DP op} is of $\ord ( XUE)$; see Remark~\ref{rem:complexity d-DP}. 
Comparing this complexity to the one reported in Theorem~\ref{thm:complexity d-CDP Alg 1}, 
points to a basic characteristic of the proposed approach: 
CDP avoids the minimization over the control input in DP and casts it as a simple addition in the dual domain at the expense of three conjugate transforms. 
Consequently, the time complexity is transferred from the primal input domain~$\setd{U}$ to the dual state domain~$\setg{Y}$. 
This observation implies that if $Y < UE$, then d-CDP is expected to computationally outperform d-DP. 

We now consider the error of Algorithm~\ref{alg:d-CDP general} w.r.t. the DP operator~\eqref{eq:DP op}. 
Let us begin with presenting an alternative representation of the d-CDP operator that sheds some light on the main sources of error. 

\begin{Prop}[d-CDP reformulation] \label{prop:d-CDP op}
Assume that the stage cost $\cost: \setc{X} \times \setc{U} \ra \Ru$ is convex in the input variable. 
Then, the d-CDP operator~\eqref{eq:d-CDP op} equivalently reads as
\begin{equation} \label{eq:d-CDP op alt}
\dcdpo [\disc{J}](x) = \min_{u} \left\{ \cost(x,u) + \bcdd{J} \big( \dyn(x,u) \big) \right\}, \quad x \in \setg{X},
\end{equation}
where $\bcdd{J}$ is the discrete biconjugate of $J$, using the primal grid $\setg{X}$ and the dual grid $\setg{Y}$.
\end{Prop}

\begin{proof}
See Appendix~\ref{proof:d-CDP op}.
\end{proof}

Comparing the representations~\eqref{eq:DP op} and \eqref{eq:d-CDP op alt}, we note that the d-CDP operator~$\dcdpo$ differs from the DP operator~$\dpo$ in that it uses $\bcdd{J}$ as an approximation of $J$. 
This observation points to two main sources of error in the proposed approach, namely, dualization and discretization. 
Indeed, $\dcdpo$ is a discretized version of the dual problem~\eqref{eq:CDP op}. 
Regarding the dualization error, we note that the d-CDP operator is ``blind'' to non-convexity; 
that is, it essentially replaces the cost-to-go~$J$ by its convex envelope (the greatest convex function that supports $J$ from below). 
The discretization error, on the other hand, depends on the choice of the finite primal and dual domains $\setg{X}$ and $\setg{Y}$. 
In particular, by a proper choice of $\setg{Y}$, it is indeed possible to eliminate the corresponding error due to discretization of the dual domain. 
To illustrate, let $\disc{J}$ be a one-dimensional, discrete, convex-extensible function with domain~$\setg{X} = \{x^i\}_{i=1}^N \subset \R$, where $x^i < x^{i+1}$. (Recall that by convex-extensible, we mean that $\disc{J}$ can be extended to convex function $\lerp{\disc{J}}$ such that $\lerp{\disc{J}}(x) = \disc{J}(x)$ for all $x \in \setg{X}$).
Also, choose $\setg{Y} = \{y^i\}_{i=1}^{N-1} \subset \R$ with $y^i = \frac{\disc{J}(x^{i+1})-\disc{J}(x^i)}{x^{i+1}-x^i}$ as the discrete dual domain. 
Then, for all $x \in \co (\setg{X}) = [x^1,x^N]$, we have $\bcdd{J}(x) = \llerp{\disc{J}}(x)$, i.e., the LERP extension. 
Hence, the only source of error under this choice of $\setg{Y}$ is the discretization of the primal state space (i.e., approximation of the true $J$ via $\llerp{\disc{J}}$). 
However, a similar construction of $\setg{Y}$ in dimensions $n \geq 2$ can lead to dual grids of size $Y = \ord(X^n)$, which is computationally impractical; see Theorem~\ref{thm:complexity d-CDP Alg 1}. 
The following result provides us with specific bounds on the discretization error that point to a more practical way for construction of $\setg{Y}$.

\begin{Thm}[Error of d-CDP Algorithm~\ref{alg:d-CDP general}] \label{thm:error d-CDP Alg 1}
Consider the DP operator~$\dpo$~\eqref{eq:DP op} and the implementation of the d-CDP operator~$\dcdpo$~\eqref{eq:d-CDP op} via Algorithm~\ref{alg:d-CDP general}. 
Assume that $\cost: \setc{X} \times \setc{U} \ra \Ru$ is convex in the input variable. 
Also, assume that $J:\setc{X} \ra \R$ is a Lipschitz continuous, convex function. 
Then, for each $x \in \setg{X}$, it holds that
\begin{equation} \label{eq:err bound d-CDP Alg 1}
- e_2 \leq \dpo [J](x) - \dcdpo [\disc{J}](x) \leq e_1(x),
\end{equation}
where
\begin{eqnarray} \label{eq:err terms d-CDP Alg 1}
\begin{array}{l}
e_1(x) = \big[ \norm{\dynx(x)} + \norm{\dynu(x)} \cdot \diam{\setc{U}} + \diam{\setc{X}} \big] \cdot \dist\big(\partial  \dpo [J] (x),\setg{Y}\big), \\
e_2 =   \left[ \diam{\setg{Y}} +  \lip(J) \right] \cdot \dish (\setc{X}, \setg{X}).
\end{array}
\end{eqnarray}
\end{Thm}

\begin{proof}
See Appendix~\ref{proof:error d-CDP Alg 1}.
\end{proof}

Notice how the two terms $e_1$ and $e_2$ capture the errors due to the discretization of the dual state space ($\setc{Y}$) and the primal state space ($\setc{X}$), respectively. 
In particular, the first error term suggests that we choose $\setg{Y}$ such that $\setg{Y} \cap \partial  \dpo [J] (x)  \neq \emptyset$ for all $x\in \setg{X}$. 
Even if we had access to $\dpo [J]$, satisfying such a condition could again lead to dual grids of size $Y = \ord(X^n)$. 
A more realistic objective is then to choose~$\setg{Y}$ such that $\co (\setg{Y}) \cap \partial  \dpo [J] (x) \neq \emptyset$ for all $x \in \setg{X}$. 
With such a construction, $\dist\big(\partial  \dpo [J] (x),\setg{Y}\big)$ and hence $e_1$ decrease by using finer grids for the dual domain. 
The latter condition is satisfied if $\co(\setg{Y}) \supseteq \setc{L} (\dpo [J])$.
Hence, we need to approximate ``the range of slopes'' of the function $\dpo [J]$ for $x \in \setg{X}$. 
Notice, however, that we do not have access to $\dpo [J]$ since it is the \emph{output} of the d-CDP operation in Algorithm~\ref{alg:d-CDP general}. 
What we have at our disposal as \emph{inputs} are the stage cost~$\cost$ and the next step (discrete) cost-to-go~$\disc{J}$. 
A coarse way to approximate the range of slopes of $\dpo [J]$ is then to use the extrema of the functions $\cost$ and $\disc{J}$, and the diameter of $\setg{X}$ in each dimension. 
The following remark explains such an approximation for the construction of $\setg{Y}$. 

\begin{Rem}[Construction of $\setg{Y}$]\label{rem:constr Y}
Let $\cost^M = \max_{x,u} \cost(x,u)$ and $\cost^m = \min_{x,u} \cost(x,u)$. 
Compute $J^M = \max_{x \in \setg{X}} \disc{J}(x)$ and $J^m = \min_{x \in \setg{X}} \disc{J}(x)$, 
and then choose $\setg{Y} = \Pi_{i=1}^{n} \setg{Y}_{i} \subset \R^n$ such that for each dimension $i =1,\ldots,n$, 
we have $$ \pm \alpha \cdot \frac{\cost^M+J^M-\cost^m-J^m}{\diam{\setg{X}_i}} \in \co(\setg{Y}_i).$$ 
Above, $\alpha > 0$ is a scaling factor mainly depending on the dimension $n$ of the state space. 
Such a construction of $\setg{Y}$ requires $\ord (X)$ operations \emph{per iteration} for computing $J^M$ and $J^m$ via enumeration over $x \in \setg{X}$.
\end{Rem}
  
%===================================================================================
\section{From quadratic complexity to linear complexity} \label{sec:reduced complexity}
%===================================================================================

In this section, we focus on a specific subclass of the optimal control problems considered in this study. 
In particular, we exploit the problem structure in this subclass to reduce the computational cost of the d-CDP algorithm. 
In this regard, a closer look to Algorithm~\ref{alg:d-CDP general} reveals a computational bottleneck in its numerical implementation: The computation of the objects $\disc{\psi}_x : \setg{Y} \ra \R,\ x \in \setg{X}$, and their conjugates which requires working in the product space $\setg{X} \times \setg{Y}$. 
This step is indeed the dominating factor in the time complexity of $\ord (XY)$ of Algorithm~\ref{alg:d-CDP general}; see Appendix~\ref{proof:complexity d-CDP Alg 1} for the proof of Theorem~\ref{thm:complexity d-CDP Alg 1}. 
Hence, if the structure of the problem allows for the complete decomposition of these objects, then a significant reduction in the time complexity is achievable. 
This is indeed possible for problems with separable data:  

\begin{Set}\label{Set:prob class II} 
\textbf{(i)}~The dynamics is input-affine with state-independent input dynamics, i.e.,  
 $\dyn(x,u) = \dynx(x) + B \cdot u$, where $\dynx : \R^n \ra \R^n$ and $B \in \R^{n \times m}$. 
\textbf{(ii)}~The stage cost is separable in state and input, i.e., $\cost(x,u) = \costx(x)+\costu(u)$, where $\costx: \setc{X} \ra \R$ and $\costu: \setc{U} \ra \R$ are the state and input costs, respectively. 
\end{Set}

Note that the separability of the stage cost $C$ implies that the constraints are also separable, i.e, there are no state-dependent input constraints. 

\subsection{Modified d-CDP algorithm} \label{subsec:d-CDP algo 2}

For the separable cost of Setting~\ref{Set:prob class II}, the state cost ($\costx$) can be taken out of the minimization in the DP operator~\eqref{eq:DP op} as follows
\begin{align} \label{eq:DP op separ}
\dpo [J](x) = \costx(x) + \min_{u}\left\{ \costu(u) + J \big(\dyn(x,u)\big) \right\}, \quad x \in \setc{X}.
\end{align}
Following the same dualization and then discretization procedure described in Section~\ref{subsec:d-CPP alg 1}, we can derive the corresponding d-CDP operator
\begin{subequations} \label{eq:d-CDP op separ}
\begin{align} 
&\lftdd{J} (y) = \max_{x \in \setg{X}}\left\{\inner{y}{x} - \disc{J}(x) \right\}, & y \in \setg{Y}, \\
& \disc{\psi}(y) \Let \lftc{\costu}(-B\tr y) + \lftdd{J}(y), & y \in \setg{Y}, \label{eq:d-CDP op separ b} \\
& \dcdpo [\disc{J}](x) = \costx(x) + \lftd{\psi} \big( \dynx(x) \big), &  x \in \setg{X}. \label{eq:d-CDP op separ a}
\end{align}
\end{subequations}
Here, again, we assume that the conjugate of the input cost is analytically available (similar to Assumption~\ref{As:conj cost func alg 1}, now in the context posed by Setting~\ref{Set:prob class II}). 
It is also possible to compute this object numerically; see Appendix~\ref{subsec:extension num conj} for more details. 

\begin{As}[Conjugate of input cost]\label{As:conj cost func alg 2}
The conjugate function $\lftc{\costu}(v) = \max_{u } \{\inner{v}{u} - \costu(u) \}$ is analytically available; that is, the complexity of evaluating $\lftc{\costu}(v)$ for each $v \in \R^m$ is of $\ord (1)$.
\end{As} 

Notice how the function~$\disc{\psi}$ in~\eqref{eq:d-CDP op separ b} is now independent of the state variable~$x$. 
This means that the computation of $\disc{\psi}$ requires $\ord (X+Y)$ operations, 
as opposed to $\ord (XY)$ for the computation of $\disc{\psi}_x$ in Algorithm~\ref{alg:d-CDP general}. 
What remains to be addressed is the computation of the conjugate function~$\lftd{\psi} \big( \dynx(x) \big) = \max_{y \in \setg{Y}}\{\inner{\dynx(x)}{y} - \psi(y) \}$ for $x \in \setg{X}$ in~\eqref{eq:d-CDP op separ a}. 
The straightforward maximization via enumeration over $y \in \setg{Y}$ for each $x \in \setg{X}$ (as in Algorithm~\ref{alg:d-CDP general}) again leads to a time complexity of $\ord (XY)$. 
The key idea here is to use \emph{approximate discrete conjugation}:
\begin{itemize} 
\item \textbf{Step 1.} Use LLT to compute $\lftdd{\psi}: \setg{Z} \ra \R$ from the data points $\disc{\psi}: \setg{Y} \ra \R$ for a grid $\setg{Z}$;
\item \textbf{Step 2.} For each $x \in \setg{X}$, use LERP to compute $\llerp{\lftdd{\psi}}\big( \dynx(x) \big)$ from the data points $\lftdd{\psi}: \setg{Z} \ra \R$. 
\end{itemize}
Proper construction of the grid $\setg{Z}$ will be discussed in the next subsection. 
With such an approximation, the d-CDP operator~\eqref{eq:d-CDP op separ} \emph{modifies} to  
\begin{subequations} \label{eq:d-CDP op separ modified}
\begin{align}
&\lftdd{J} (y) = \max_{x \in \setg{X}}\left\{\inner{y}{x} - \disc{J}(x) \right\}, & y \in \setg{Y}, \\
& \disc{\psi}(y) = \lftc{\costu}(-B\tr y) + \lftdd{J}(y), & y \in \setg{Y}, \label{eq:d-CDP op separ modeified c} \\
& \lftdd{\psi} (z) = \max_{y \in \setg{Y}}\left\{\inner{z}{y} - \disc{\psi}(y) \right\}, & z \in \setg{Z}, \label{eq:d-CDP op separ modeified b} \\
& \mdcdpo [\disc{J}](x) \Let \costx(x) + \llerp {\lftdd{\psi}} \big( \dynx(x) \big). &  x \in \setg{X}, \label{eq:d-CDP op separ modeified a}
\end{align}
\end{subequations}
Algorithm~\ref{alg:d-CDP separ} provides the pseudo-code for the numerical scheme described above. 

\begin{algorithm}[t]
\begin{small}
   \caption{Implementation of the modified d-CDP operator~\eqref{eq:d-CDP op separ modified} for Setting~\ref{Set:prob class II}.}
   \label{alg:d-CDP separ}
\begin{algorithmic}[1]
	\REQUIRE dynamics~$\dynx: \R^n \ra \R^n, \ B \in \R^{n\times m}$; \\
	discrete cost-to-go (at $t+1$)~$\disc{J}: \setg{X} \ra \R$; \\
	state cost~$\costx(x): \setc{X} \ra \R$; \\
	conjugate of input cost~$\lftc{\costu}: \R^m \ra \R$; \\
	grids~$\setg{Y}, \setg{Z} \subset \R^n$.
	\ENSURE discrete cost-to-go (at $t$)~$\mdcdpo [\disc{J}](x): \setg{X} \ra \R$.
  	
  	\vspace{.3cm}
  	
    \STATE use LLT to compute $\lftdd{J}:\setg{Y} \ra \R$ from $\disc{J}: \setg{X} \ra \R$; \label{line_alg_s:LLT of J}
    
	\vspace{.1cm}
	
	\STATE $\disc{\psi}(y) \gets \lftc{\costu}(-B\tr y) + \lftdd{J}(y)$ for $y \in \setg{Y}$; \label{line_alg_s:h}
    
	\vspace{.1cm}    
    
    \STATE use LLT to compute $\lftdd{\psi}: \setg{Z} \ra \R$ from  $\disc{\psi}: \setg{Y} \ra \R$; \label{line_alg_s:LLT of h}
    
    \vspace{.1cm}
    
    \FOR{each $x \in \setg{X}$} 
    
 		\STATE  use LERP to compute $\llerp{\lftdd{\psi}}\big( \dynx(x) \big)$ from $\lftdd{\psi}: \setg{Z} \ra \R$; \label{line_alg_s:LERP of h}
     	
    	\STATE $\mdcdpo [\disc{J}](x) \gets \costx (x) + \llerp{\lftdd{\psi}}\big( \dynx(x) \big)$; \label{line_alg_s:TJ}
    	
   	\ENDFOR	
   	
\end{algorithmic}
\end{small}
\end{algorithm}

\subsection{Analysis of modified d-CDP algorithm} \label{subsec: analysis of Alg 2}

We again begin with the time complexity of the proposed algorithm.

\begin{Thm} [Complexity of modified d-CDP Algorithm~\ref{alg:d-CDP separ}] \label{thm:complexity d-CDP Alg 2}
Let Assumptions~\ref{As:grid size LLT} and ~\ref{As:conj cost func alg 2} hold. 
Then, the computation of the modified d-CDP operator~\eqref{eq:d-CDP op separ modified} via Algorithm~\ref{alg:d-CDP separ} has a time complexity of $\wt{\ord} (X + Y + Z)$. 
\end{Thm}

\begin{proof}
See Appendix~\ref{proof:complexity d-CDP Alg 2}.
\end{proof}

Comparing the time complexity of the d-CDP Algorithm~\ref{alg:d-CDP separ} with that of the d-CDP Algorithm~\ref{alg:d-CDP general}, 
we observe a reduction from quadratic complexity to (log-)linear complexity. 
In particular, if all the involved grids ($\setg{X},\setg{Y},\setg{Z}$) are of the same size, i.e., $Y,Z = X$ (this is also consistent with Assumption~\ref{As:grid size LLT}), 
then the complexity of the d-CDP Algorithm~\ref{alg:d-CDP general} is of $\ord(X^2)$, 
while that of the d-CDP Algorithm~\ref{alg:d-CDP separ} is of $\wt{\ord}(X)$.

We next consider the error of the proposed algorithm by providing a bound on the difference between the modified d-CDP operator~\eqref{eq:d-CDP op separ modified} and the DP operator~\eqref{eq:DP op separ}.

\begin{Thm}[Error of modified d-CDP Algorithm~\ref{alg:d-CDP separ}]\label{thm:error d-CDP Alg 2}
Consider the DP operator~$\dpo$~\eqref{eq:DP op separ} and the implementation of the modified d-CDP operator~$\mdcdpo$~\eqref{eq:d-CDP op separ modified} via Algorithm~\ref{alg:d-CDP separ}. 
Assume that the input cost $\costu: \setc{U} \ra \R$ is convex, 
and the function $J:\setc{X} \ra \R$ is a Lipschitz continuous, convex function. 
Also, assume that the grid~$\setg{Z}$ in Algorithm~\ref{alg:d-CDP separ} is such that $\co (\setg{Z}) \supseteq \dynx(\setg{X})$. 
Then, for each~$x \in \setg{X}$, we have
\begin{equation} \label{eq:err bound d-CDP Alg 2}
-\big( e_2 + e_3 \big) \leq \dpo [J](x) - \mdcdpo [\disc{J}](x) \leq e^m_1(x),
\end{equation}
where
\begin{eqnarray} \label{eq:err terms d-CDP Alg 2}
\begin{array}{l}
e^m_1(x) \Let \big[ \norm{\dynx(x)} + \norm{B} \cdot \diam{\setc{U}} + \diam{\setc{X}} \big] \cdot \dist\big(\partial  \big(\dpo [J]-\costx\big) (x),\setg{Y}\big), \\
e_2 =   \left[ \diam{\setg{Y}} +  \lip(J) \right] \cdot \dish (\setc{X}, \setg{X}), \\
e_3 = \diam{\setg{Y}} \cdot \dish\big(\dynx (\setg{X}), \setg{Z} \big).
\end{array}
\end{eqnarray}
\end{Thm} 

\begin{proof}
See Appendix~\ref{proof:error d-CDP Alg 2}.
\end{proof}

Once again, the three terms capture the errors due to discretization of $y$, $x$, and $z$, respectively. 
We now use this result to provide some guidelines on the construction of the required grids. 
Concerning the grid $\setg{Y}$, because of the error term $e^m_1$, similar guidelines to the ones provided preceding to and in Remark~\ref{rem:constr Y} apply here. 
In particular, notice that the first error term $e^m_1$ \eqref{eq:err terms d-CDP Alg 2} now depends on $\dist\big(\partial  \big(\dpo [J]-\costx\big) (x),\setg{Y}\big)$, and hence in the construction of $\setg{Y}$, we need to consider the range of slopes of $\dpo [J]-\costx$. This essentially means using $\costu^M = \max_{u \in \setc{U}} \costu$ and $\costu^m = \min_{u \in \setc{U}} \costu$ instead of $\cost^M$ and $\cost^m$, respectively, in Remark~\ref{rem:constr Y}. 

Next to be addressed is the construction of the grid~$\setg{Z}$. 
Here, we are dealing with the issue of constructing the dual grid for approximate discrete conjugation.
Then, by Corollary~\ref{cor:disc conj via lerp}, we can either 
construct a \emph{fixed} grid $\setg{Z}$ such that $\co (\setg{Z}) \supseteq \dynx(\setg{X})$, or 
construct $\setg{Z}$ \emph{dynamically} such that $\co (\setsg{Z}) \supseteq \setc{L} (\disc{\psi})$ in each iteration. 
The former has a \emph{one-time} computational cost of $\ord (X)$, while the latter requires $\ord(Y)$ operations \emph{per iteration}. 
For this reason, as also assumed in Theorem~\ref{thm:error d-CDP Alg 2}, we use the first method to construct $\setg{Z}$. 
The following remark summarizes this discussion. 

\begin{Rem}[Construction of $\setg{Z}$]\label{rem:constr Z}
Construct the grid $\setg{Z}$ such that $\co (\setg{Z}) \supseteq \dynx(\setg{X})$. 
This can be done by finding the vertices of the smallest hyper-rectangle that contains the set $\dynx\big(\setg{X}\big)$. 
Such a construction has a \emph{one-time} computational cost of~$\ord (X)$.
\end{Rem}

We finish this section with some remarks on using the output of the backward value iteration for finding a suboptimal control sequence~$\mathbf{u}\opt (x_0)$ for a given instance of the optimal control problem with initial state~$x_0$.\footnote{
We note that the backward value iteration using the d-DP algorithm also provides us with control laws $\disc{\mu}_t: \setg{X} \ra \setg{U},\ t = 0,1, \ldots, T-1$. 
However, the d-CDP algorithms only provide us with the costs $\disc{J}_t, \ t = 0,1,\ldots,T-1$.  
Hence, when the d-DP algorithm is used, we can alternatively use the control laws, accompanied by a proper extension operator, to produce a suboptimal control sequence, i.e.,
\begin{equation*}
u\opt_t(x_t) = \lerp{\disc{\mu_t}}(x_t), \quad t = 0, 1, \ldots, T-1.
\end{equation*}
This method has a time complexity of $\ord(E)$, where $E$ represents the complexity of the extension operation used above. 
This complexity can be particularly lower than that of generating greedy actions w.r.t. the computed costs in \eqref{eq:opt control J}. 
However, generating control actions using the control laws has a higher memory complexity for systems with multiple inputs, and is also usually more sensitive to modeling errors due to its completely open-loop nature.    
Moreover, we note that the \emph{total} time complexity of solving an instance of the optimal control problem, 
i.e., backward iteration for computing the costs~$\disc{J}_t$ and control laws~$\disc{\mu}_t$, and forward iteration for computing the control sequence~$\mathbf{u}\opt (x_0)$, 
is in both methods of $\ord(TXUE)$. 
That is, computationally, the backward value iteration is the dominating factor. 
} 
Having the discrete costs-to-go $\disc{J}_t:\setg{X} \ra \R, \ t = 0,1,\ldots,T-1$, at our disposal (the output of the d-DP or d-CDP algorithm), at each time step, we can use \emph{the greedy action} w.r.t. the next step's cost-to-go, i.e.,
\begin{equation}\label{eq:opt control J}
u\opt_t \in \argmin_{u_t \in \setd{U}} \left\{\cost(x_t,u_t) + \lerp{\disc{J_{t+1}}}\big(\dyn(x_t,u_t)\big) \right\}, \quad t = 0,1, \ldots, T-1,
\end{equation}
for a proper discrete input space $\setd{U}$. 
Assuming these minimization problems are handled via enumeration, 
they lead to an additional computational burden of $\ord(UE)$ per iteration, where $E$ represents the complexity of the extension operation in~\eqref{eq:opt control J}. 
Then, the \emph{total} time complexity of solving a $T$-step problem (i.e., the time requirement of backward value iteration for finding $\disc{J}_t, \ t = 0,1,\ldots,T-1$, plus the time requirement of forward iteration for finding $u\opt_t, \ t = 0,1,\ldots,T-1$) of the three algorithms can be summarized as follows.

\begin{Rem}[Comparison of total complexities]\label{rem:complexity compare}
The total time complexity of solving a $T$-step problem for a given initial state, where the control input is generated using the greedy policy~\eqref{eq:opt control J},  is of 
\begin{itemize}
\item[(i)] $\ord(TXUE)$ for the d-DP algorithm,
\item[(ii)] $\ord \big(T(XY+UE)\big)$ for the d-CDP Algorithm~\ref{alg:d-CDP general},
\item[(iii)] $\wt{\ord} \big(T(X+Y+Z+UE)\big)$ for the d-CDP Algorithm~\ref{alg:d-CDP separ}, 
\end{itemize} 
where $E$ represents the complexity of the extension operation in~\eqref{eq:d-DP op} and~\eqref{eq:opt control J}.
\end{Rem} 

Once again, we see a reduction from quadratic to linear complexity in the modified d-CDP Algorithm~\ref{alg:d-CDP separ} compared to both the d-DP algorithm and the d-CDP Algorithm~\ref{alg:d-CDP general}.

%===============================================================================
\section{Numerical experiments}
\label{sec:numerical ex}
%===============================================================================

In this section, we examine the performance of the proposed d-CDP algorithms (referred to as d-CDP~\ref{alg:d-CDP general} and d-CDP~\ref{alg:d-CDP separ} in this section) in comparison with the generic d-DP algorithm (referred to as d-DP in this section) through a synthetic numerical example. 
In particular, we use this numerical example to verify our theoretical results on the complexity and error of the proposed algorithms. 
Here, we focus on the performance of the basic algorithms for deterministic systems for which the conjugate of the (input-dependent) stage cost is analytically available (see Assumptions~\ref{As:conj cost func alg 1} and \ref{As:conj cost func alg 2}). 
The extension of the proposed d-CDP algorithms and their numerical simulations are provided in Appendix~\ref{app:num example}. 
Finally, we note that all the simulations presented in this article were implemented via MATLAB version R2017b, on a PC with an Intel Xeon 3.60~GHz processor and 16~GB RAM. 

We consider a linear system with two states and two inputs described by 
\begin{equation*}
x_{t+1} =  \left[ \begin{array}{cc} -0.5 & 2 \\ 1 & 3 \end{array} \right] x_t 
+ \left[ \begin{array}{cc} 1 & 0.5 \\ 1 & 1 \end{array} \right] u_t,
\end{equation*} 
over the finite horizon $T=10$, 
with the following state and input constraints 
\begin{equation*}
x_t \in \setc{X} = [-1,1]^2 \subset \R^2 , \quad u_t \in \setc{U} = [-2,2]^2 \subset \R^2.
\end{equation*}
Moreover, we consider \emph{quadratic state cost} and \emph{exponential input cost} as follows
\begin{equation*}
\costx (x) = \cost_T (x) =  \norm{x}^2, \quad \costu (u) = e^{|u_1|} + e^{|u_2|} - 2.
\end{equation*}
Note that the conjugate of the input cost is indeed analytically available and given by
\begin{equation*}
\lftc{\costu}(v) = 1 + \inner{\hat{u}}{v} - e^{|\hat{u}_1|} - e^{|\hat{u}_2|}, \quad v \in \R^2,
\end{equation*} 
where 
$$
\hat{u}_i = \left\{ \begin{array}{lc}
\max \big\{ -2, \ \min \left\{ 2, \ \sgn (v_i) \ln |v_i| \right\} \big\}, & v_i \neq 0, \\
0, & v_i = 0,
\end{array} 
\right.
\quad i=1,2.
$$
Moreover, corresponding to the notation of Section~\ref{sec:conj DP}, the stage cost and its conjugate are given by
\begin{align*}
\cost_x(u) &= \cost (x,u) = \norm{x}^2 + e^{|u_1|} + e^{|u_2|} - 2, \quad (x,u) \in \setc{X} \times \setc{U}, \\
\lftc{\cost_x}(v) &= \lftc{\costu}(v) - \norm{x}^2,  \quad (x,v) \in \setc{X} \times \R^2.
\end{align*}

We use \emph{uniform} grid-like discretizations $\setg{X}$ and $\setg{U}$ for the state and input spaces, such that $\co (\setg{X}) = \setc{X}$ and $\co (\setg{U}) = \setc{U}$. 
The grids $\setg{Y}$ and $\setg{Z}$ involved in d-CDP algorithms are also constructed \emph{uniformly}, according to the guidelines provided in Remarks~\ref{rem:constr Y} and \ref{rem:constr Z} (with $\alpha = 1$). 
We are particularly interested in the performance (error and time complexity) of d-CDP algorithms in comparison with d-DP, as the size of these discrete sets increases. 
Considering the fact that all the discrete sets are uniform grids, and we use LERP for all the extension operations (particularly, for the extension of the discrete cost functions in the d-DP operation~\eqref{eq:d-DP op} and for generating the greedy control actions in~\eqref{eq:opt control J}), the complexity of a single evaluation of all extensions is of $\ord(E) = \ord(1)$; see Remark~\ref{rem:complexity of LERP}. 

We begin with examining the error in d-DP and d-CDP algorithms w.r.t. the ``reference'' costs-to-go $J_t^{\star}: \setc{X} \ra \R$. 
Since the problem does not have a closed-form solution, these reference costs $J_t^{\star}$ are computed numerically via a high-resolution application of d-DP with $X,U = 81^2$.
Figure~\ref{fig:error} depicts the maximum absolute error in the discrete cost functions~$\disc{J}_t$ computed using these algorithms over the horizon.
As expected and in line with our error analysis (Theorems~\ref{thm:error d-CDP Alg 1} and \ref{thm:error d-CDP Alg 2} and Proposition~\ref{prop:error d-DP}), using a finer discretization scheme with larger $X,U,Y,Z = N$, leads to a smaller error. 
Moreover, over the time steps in the backward iteration, a general increase is seen in the error which is due to the accumulation of error.  
For further illustration, Figure~\ref{fig:J1} shows the corresponding costs-to-go at $t=9$ and $t=0$, with $N=21^2$.  
Notice that, since the stage and terminal costs are convex and the dynamics is linear, the costs-to-go are also convex. 
As can be seen in Figure~\ref{fig:J1}, while d-CDP~\ref{alg:d-CDP general} preserves the convexity, d-DP and d-CDP~\ref{alg:d-CDP separ} output non-convex costs-to-go (due to the application of LERP in these algorithms).
In particular, notice how $J_0^{\text{CDP1}}$ is convex-extensible, while $J_0^{\text{DP}}$ and $J_0^{\text{CDP2}}$ are not.

\begin{figure}[t]
\begin{subfigure}{.5\textwidth}
  \centering
  \includegraphics[clip, trim=.8cm 0cm 6.6cm 0cm,width=.6\linewidth]{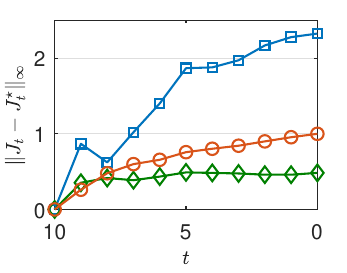}
  \caption{}
\end{subfigure}%
\begin{subfigure}{.5\textwidth}
  \centering
  \includegraphics[clip, trim=6.6cm 0cm .8cm 0cm,width=.6\linewidth]{error.pdf}
  \caption{}
\end{subfigure}
  \caption{Error of the computed discrete costs $\disc{J}_t:\setg{X} \ra \R$ using d-DP, d-CDP~\ref{alg:d-CDP general}, and d-CDP~\ref{alg:d-CDP separ} for grid sizes $X,U,Y,Z=N$: (a) $N=11^2$; (b) $N=41^2$. Notice that the time axis is backward.}
  \label{fig:error}
\end{figure}

\begin{figure}[t]
  \centering
  \includegraphics[clip, trim=1.5cm .5cm 1cm .5cm,width=1\linewidth]{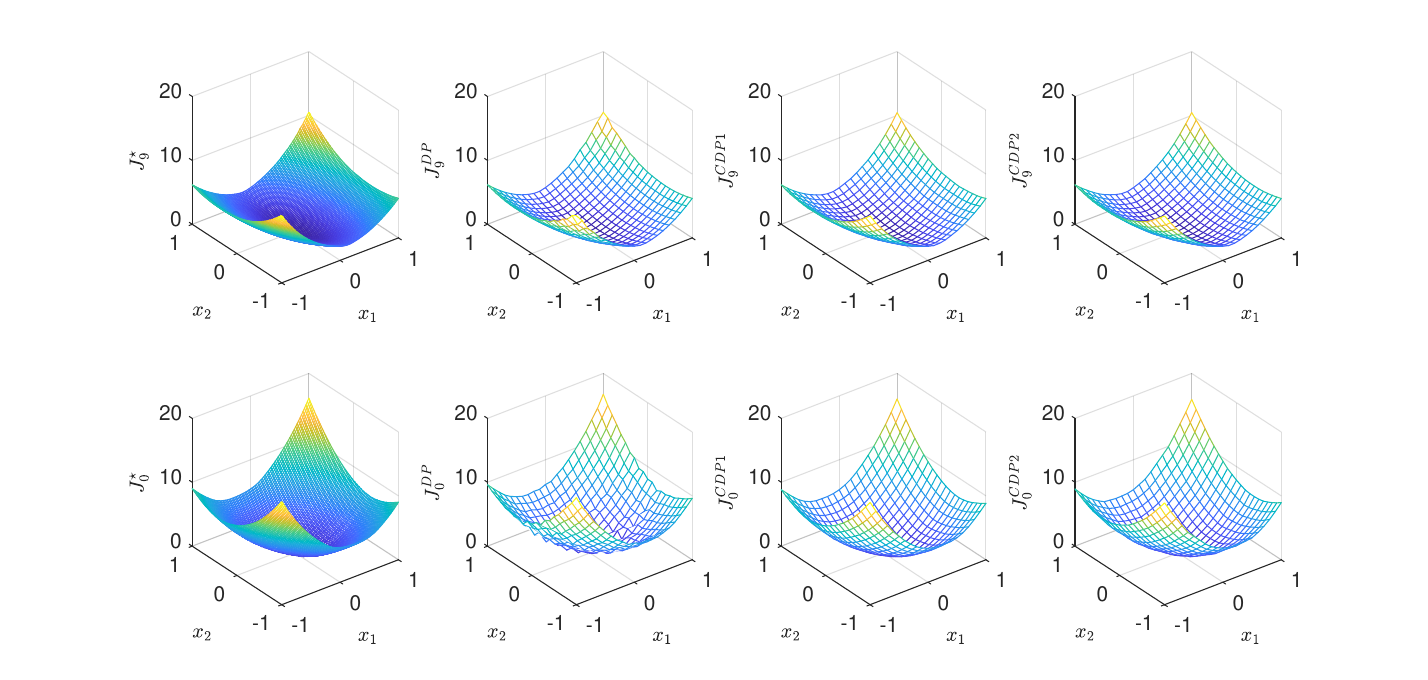}
  \caption{Computed discrete costs $\disc{J}_t:\setg{X} \ra \R$ using d-DP, d-CDP~\ref{alg:d-CDP general}, and d-CDP~\ref{alg:d-CDP separ} for grid sizes $X,U,Y,Z=21^2$ at $t=9$ (top) and $t=0$ (bottom).}
  \label{fig:J1}
\end{figure} 

We next compare the performance of the three algorithms in solving instances of the optimal control problem, using the cost functions derived from the backward value iteration. 
To this end, we apply the greedy control input~\eqref{eq:opt control J} w.r.t. the computed discrete costs-to-go $\disc{J}_t$ using d-DP and d-CDP algorithms, and the same discrete input space $\setg{U}$ as the one in d-DP. 
Let us first consider the complexity of d-DP and d-CDP algorithms. 
Figure~\ref{fig:complexity} reports the \emph{total} run-time of a random problem instance for different grid sizes (i.e., the time requirement of backward value iteration for finding $\disc{J}_t, \ t = 0,1,\ldots,T-1$, plus the time requirement of forward iteration for finding $u\opt_t, \ t = 0,1,\ldots,T-1$). 
Regarding the reported running times, note that they correspond to the given complexities in Theorems~\ref{thm:complexity d-CDP Alg 1} and \ref{thm:complexity d-CDP Alg 2} and Remark~\ref{rem:complexity compare}: 
For our numerical example, the running time is of $\ord(TN^2)$ for d-DP and d-CDP~\ref{alg:d-CDP general}, and of $\ord(TN)$ for d-CDP~\ref{alg:d-CDP separ}. 
The difference can be readily seen in the slope of the corresponding lines in Figure~\ref{fig:complexity} as $N$ increases. 
In this regard, we also note that the backward value iteration is the absolutely dominant factor in the reported running times. (Effectively, the reported numbers can be taken to be the run-time of the backward value iteration). 
In Figure~\ref{fig:cost}, we also report the average cost of the controlled trajectories over 100 instances of the optimal control problem with random initial conditions, chosen uniformly from $\setc{X} = [-1,1]^2$. 

\begin{figure}[t]
\begin{subfigure}{.5\textwidth}
  \centering
  \includegraphics[clip, trim=0cm 0cm 0cm .2cm,width=.6\linewidth]{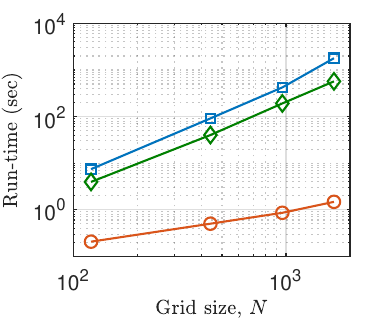}
  \caption{}
  \label{fig:complexity}
\end{subfigure}%
\begin{subfigure}{.5\textwidth}
  \centering
  \includegraphics[clip, trim=0cm 0cm 0cm .2cm,width=.6\linewidth]{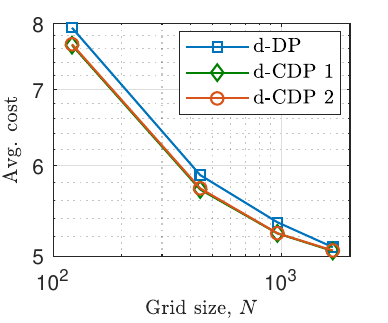}
  \caption{}
  \label{fig:cost}
\end{subfigure}
\caption{Performance of d-DP, d-CDP~\ref{alg:d-CDP general}, d-CDP~\ref{alg:d-CDP separ} for different grid sizes $X,U,Y,Z = N$: (a) the total running time for solving a random problem instance; (b) the average cost of controlled trajectories for 100 random initial states.}
\label{fig:perf}
\end{figure}

Looking at Figure~\ref{fig:perf}, one notices that d-CDP~\ref{alg:d-CDP separ}, compared to d-DP, has a similar performance when it comes to the quality of greedy control actions, however, with a significant reduction in the running time. 
In particular, notice how the lower complexity of d-CDP~\ref{alg:d-CDP separ} allows us to increase the size of the grids to $N=41^2$, while keeping the running time at the same order as that of d-DP with $N=11^2$. 
Comparing the performance of d-CDP~\ref{alg:d-CDP general} with d-DP, on the other hand, one notices that they show effectively the same performance w.r.t. the considered measures. 
d-CDP~\ref{alg:d-CDP general}, however, gives us an extra degree of freedom for the size $Y$ of the dual grid. 
In particular, if the cost functions are ``compactly representable'' in the dual domain (i.e., via their slopes), we can reduce the time complexity of d-CDP~\ref{alg:d-CDP general} by using a more coarse grid $\setg{Y}$, with a limited effect on the ``quality'' of computed cost functions. 
This effect is illustrated in Table~\ref{tab:comparison 1}: For solving the same optimal control problem with $X,U,Y=41^2$, we can reduce the size of the dual grid by a factor of $4$ to $Y=21^2$, and hence reduce the running time of d-CDP~\ref{alg:d-CDP general}, while achieving the same average cost in the controlled trajectories.

\begin{table}[t]
\caption{Performance d-DP and d-CDP~\ref{alg:d-CDP general} for grid sizes $X,Y,U$: 
The reported numbers are the total running time for solving a random problem instance and the average cost of controlled trajectories for 100 random initial states. The first two rows correspond to the rightmost data points in Figure~\ref{fig:perf}.}
\label{tab:comparison 1}
%\vskip 0.15in
\begin{center}
\begin{small}
%\begin{sc}
\begin{tabular}{lcc}
\toprule
Algorithm  & Run-time (sec) & Avg. cost  \\
\midrule
$\mathrm{d}$-DP with $X,U = 41^2$                                       & $1790$ & $5.09$  \\
$\mathrm{d}$-CDP~\ref{alg:d-CDP general} with $X,Y = 41^2$         & $570$ & $5.05$  \\
$\mathrm{d}$-CDP~\ref{alg:d-CDP general} with $X = 41^2, Y = 21^2$ & $187$ & $5.05$  \\
\bottomrule
\end{tabular}
%\end{sc}
\end{small}
\end{center}
%\vskip -0.1in
\end{table}

%===============================================================================
\section{Final remarks}
\label{sec:remarks}
%===============================================================================

In this final section, the limitations of the proposed algorithms and possible remedies to alleviate them are discussed.
We also discuss some of the algorithms available in the literature and their connection to the d-CDP algorithms. 
Finally, we mention possible extensions of the current work as future research directions. 

\subsection{Curse of dimensionality and grid-like discretization} 

The proposed d-CDP algorithms still suffer from the infamous ``curse of dimensionality'' in the sense that the computational cost increases exponentially with the dimension of the state and input spaces. 
This is because the size of the discretized state and input spaces increase exponentially with their dimensions. 
However, we note that in the d-CDP Algorithm~\ref{alg:d-CDP separ}, the rate of exponential increase is $\max\{n,m\}$ (corresponding to $\ord(X+U)$ complexity), compared to the rate $m+n$ for the d-DP algorithm (corresponding to $\ord(XU)$ complexity). 
Moreover, in this study, we used grid-like discretizations of \emph{both} primal and dual state spaces. 
This is particularly suitable for problems with (almost) box constraints on the state variables (as illustrated in the numerical example in Section~\ref{sec:numerical ex}). 
However, we note that to enjoy the linear-time complexity of LLT, we are \emph{only} required to choose a grid-like \emph{dual} grid \cite[Rem.~5]{Lucet97}; 
that is, the discretization of the state space in the primal domain need not be grid-like. 

\subsection{Towards quantum dynamic programming} \label{subsec:quantum CDP} 
An interesting feature of the conjugate dynamic programming framework proposed in this study is that it can be potentially combined with existing tools/techniques for further reduction in time complexity. 
For example, the proposed framework can be readily combined with sample-based value iteration algorithms that focus on transforming the infinite-dimensional optimization in DP problems into computationally tractable ones (e.g., the common state aggregation technique \cite[Sec. 8.1]{Pow11} with piece-wise constant approximation). 
More interestingly, motivated by the recent quantum speedup for discrete conjugation~\cite{Sutter20},
we envision that the proposed framework paves the way for developing a quantum DP algorithm. 
Indeed, the proposed algorithms are developed such that any reduction in the complexity of discrete conjugation immediately translates to a reduced computational cost of these algorithms. 

\subsection{Value iteration in the conjugate domain} \label{subsec:fvi}
Let us first note that the algorithms developed in this study involve two LLT transforms at the beginning and end of each step (see, e.g., lines~\ref{line_alg_s:LLT of J} and~\ref{line_alg_s:LLT of h} in Algorithm~\ref{alg:d-CDP separ}). 
Hence, the possibility of a perfect transformation of the minimization in the primal domain to a simple addition in the conjugate domain is interesting since it allows for performing the value iteration completely in the conjugate domain for the \emph{conjugate} of the costs-to-go. 
In other words, we can stay in the conjugate domain over multiple steps in time, and avoid the first conjugate operation at the beginning of the intermediate steps. 
This, in turn, leads to a lower computational cost in multistep implementations. 
However, for such a perfect transformation to be possible, we need to impose further restrictions on the problem data. 
To be precise, we need (cf. Setting~\ref{Set:prob class II}) 
\begin{itemize}
\item[(i)] the dynamics to be \emph{linear}, i.e., $\dyn(x,u) = Ax+Bu$, where the state matrix $A$ is \emph{invertible}, 
\item[(ii)] and the stage cost to be \emph{state-independent}, i.e., $\cost(x,u) = \costu(u)$.
\end{itemize}
For systems satisfying these conditions, the DP operator reads as
\begin{align*}
\dpo [J] (x)= \min_{u }\left\{ \costu(u) + J (Ax+Bu) \right\}, \quad x \in \setc{X},
\end{align*}
and its conjugate can be shown to be given by  
\begin{align*}
\lftc{[\dpo [J]]} (y) = \lftc{\costu}(-B\tr A\itr y) + \lftc{J}(A\itr y), \quad y \in \R^n. 
\end{align*}
Notice the perfect transformation of the minimization in the DP operator in the primal domain to an addition in the dual domain. 
This property indeed allows us to stay in the dual domain over multiple steps in time, while only computing the conjugate of the costs in the intermediate steps. 
The possibility of such a perfect transformation, accompanied by the application of LLT for better time complexity, was first noticed in \cite{Caprio16}, 
where the authors introduced the ``fast value iteration'' algorithm for a more restricted class of DP problems (besides the properties discussed above, they required, among other conditions, the state matrix $A$ to be non-negative and monotone). 
In this regard, we also note that, as in \cite{Caprio16}, the possibility of staying in the conjugate domain over multiple steps is particularly interesting for infinite-horizon problems.  

\subsection{Relation to max-plus linear approximations} \label{subsec:max-plus} 
Recall the d-CDP reformulation
\begin{equation*} 
\dcdpo [\disc{J}](x) = \min_{u} \left\{ \cost(x,u) + \bcdd{J} \big( \dyn(x,u) \big) \right\},
\end{equation*}
in Proposition~\ref{prop:d-CDP op}, and note that 
$$
\bcdd{J}(x) = \max_{y \in \setg{Y}} \left\{ \inner{x}{y} - \lftdd{J}(y) \right\},
$$
is a max-plus linear combination using the linear basis functions $x \mapsto \inner{x}{y}$ and coefficients $\lftdd{J}(y)$, with $y \in \setg{Y}$ being the slopes for the basis functions. 
That is, the d-CDP algorithm, similarly to the approximate value iteration algorithms in~\cite{Bach19,Bach20}, employs a max-plus linear approximation of $J$ as a piece-wise affine function. 
The key difference in our algorithms is however that by choosing a grid-like dual domain~$\setg{Y}$, we can take advantage of the linear-time complexity of LLT in computing the coefficients $\lftdd{J}(y)$ using the data points $\disc{J}:\setg{X} \ra \R$. 
Moreover, instead of using a fixed basis, we incorporate a dynamic basis by updating the grid~$\setg{Y}$ at each iteration to reduce the error of the algorithm. 

An interesting future research direction is to consider other forms of max-plus linear approximations for the cost functions. 
In particular, instead of convex, piece-wise affine approximation, one can consider the semi-concave, piece-wise quadratic approximation~\cite{McEn06} 
$$
J^{\mathrm{d}\circledast\mathrm{d}\circledast}(x) = \max_{w \in \setg{W}} \left\{ c \norm{x-w}^2 + J^{\mathrm{d}\circledast\mathrm{d}}(w) \right\},
$$
for a proper finite set $\setg{W} \subset \setc{X}$ and constant $c >0$. 
The important issue then is the fast computation of the coefficients $J^{\mathrm{d}\circledast\mathrm{d}}: \setg{W} \ra \R$ using the data points $\disc{J}:\setg{X} \ra \R$. 
This seems to be possible considering the fact that the operation $[\cdot]^{\circledast}$ closely resembles the ``distance transform'' \cite{Felzen12,Lucet09}. 

\subsection{The optimizer map in LLT} 
Consider a discrete function~$\disc{h}:\setd{X} \ra \R$ and its discrete conjugate $\lftdd{h}:\setg{Y} \ra \R$ computed using LLT for some finite set $\setg{Y}$.  
LLT is, in principle, capable of providing us with the optimizer mapping 
$$x\opt: \setg{Y} \ra \setd{X}: y \mapsto \argmax_{x \in \setd{X}} \{ \inner{x}{y} - \disc{h}(x) \},$$ 
where for each $y \in \setg{Y}$, we have $\lftdd{h}(y) = \inner{x\opt(y)}{y} - \disc{h}\big(x\opt(y)\big)$. 
This capability of LLT can be employed to address some of the drawbacks of the proposed d-CDP algorithm:

\textit{(i)~Avoiding approximate conjugation:} Let us first recall that by approximate (discrete) conjugation we mean that we first compute the conjugate function $\lftdd{h}:\setg{Y} \ra \R$ for some grid $\setg{Y}$ using the data points $\disc{h}:\setd{X} \ra \R$, 
and then for any $\tilde{y}$ (not necessarily belonging to $\setg{Y}$) we use the LERP extension $\llerp{\lftdd{h}}(\tilde{y})$ as an approximation for $\lftd{h}(\tilde{y})$. 
This approximation scheme is used in Algorithm~\ref{alg:d-CDP separ} (and all the extended algorithms in Appendix~\ref{app:ext} for computing the conjugate of the stage cost numerically). 
Indeed, it is possible to avoid this approximation and compute $\lftd{h}(\tilde{y})$ exactly by incorporating a smart search for the corresponding optimizer~$\tilde{x} \in \setd{X}$ for which $\lftd{h}(\tilde{y}) = \inner{\tilde{x}}{\tilde{y}} - h(\tilde{x})$. 
To be precise, if $\tilde{y} \in \co( \setd{\wt{Y}})$ for some subset~$\setd{\wt{Y}}$ of $\setg{Y}$, then $\tilde{x} \in \co\big( x\opt(\setd{\wt{Y}}) \big)$, where $x\opt: \setg{Y} \ra \setd{X}$ is the corresponding optimizer mapping. 
That is, in order to find the optimizer~$\tilde{x} \in \setd{X}$ corresponding to $\tilde{y}$, it suffices to search in the set~$\setd{X} \cap \co\big( x\opt(\setd{\wt{Y}}) \big)$, instead of the entire discrete primal domain~$\setd{X}$. 
This, in turn, can lead to a lower time requirement for computing the exact discrete conjugate function. 

\textit{(ii)~Extracting the optimal policy within the d-CDP algorithm:} 
The backward value iteration using the proposed d-CDP algorithms provides us \emph{only} with discrete costs $\disc{J}_t: \setg{X} \ra \R,\ t = 0,1, \ldots, T-1$. 
%On the other hand, 
On the other hand, the backward value iteration using the d-DP algorithm \emph{also} provides us with control laws $\disc{\mu}_t: \setg{X} \ra \setd{U},\ t = 0,1, \ldots, T-1$. 
Application of these control laws can potentially render the computation of the control sequence for a given initial condition less costly. 
To address this issue, we have to look at the possibility of extracting the control laws within the d-CDP algorithm. 
A promising approach is to keep track of the dual pairs in each conjugate transform, i.e., the pairs $(x,y)$ for which $\inner{x}{y} = h(x)+\lftc{h}(y)$. 
This indeed seems possible considering the capability of LLT in providing the optimizer mapping $x\opt: \setg{Y} \ra \setg{X}$.

%===============================================================================
%	\bibliographystyle{ieeetr} %{siam}
%	\bibliography{ref}
%===============================================================================

%==================================== APPENDICES ===============================
\appendix 

%=====================================================
\section{Error of d-DP} \label{app:error d-DP}

In this section, we consider the error in the d-DP operator w.r.t. the DP operator.

\begin{Prop}[Error of d-DP]\label{prop:error d-DP} 
Consider the DP operator~$\dpo$~\eqref{eq:DP op} and the d-DP operator~$\ddpo$~\eqref{eq:d-DP op}.
Assume that the functions $J$ and $\lerp{\disc{J}}$ are Lipschtiz continuous, and $\lerp{\disc{J}}(x) = J(x)$ for all $x \in \setg{X}$. 
Then, 
\begin{equation*}\label{eq:error d-DP}
-e_1 \leq \ddpo [\disc{J}](x) - \dpo [J](x) \leq e_1 + e_2(x), \quad \forall x \in \setg{X},
\end{equation*}
where
\begin{align*}
&e_1 = \big[ \lip(J) + \lip(\lerp{\disc{J}}) \big] \cdot \dish (\setc{X}, \setg{X}), \\ 
&e_2(x) = \big[ \lip(J) + \lip(C) \big] \cdot \dish \big(\setc{U}(x), \setd{U}(x)\big).
\end{align*} 
\end{Prop}

\begin{proof}
Define $ Q_x(u) \Let  \cost(x,u) + J\big(\dyn(x,u)\big)$ and $\lerp{Q}_x(u) \Let \cost(x,u) + \lerp{\disc{J}}\big(\dyn(x,u)\big)$. 
Let us fix $x \in \setg{X}$. 
In what follows, we consider the effect of (i)~replacing $J$ with $\lerp{\disc{J}}$, and (ii)~minimizing over $\setd{U}$ instead of $\setc{U}(x)$, separately. 
To this end, we define the \emph{intermediate} DP operator
\begin{align*} 
& \idpo [J](x) \Let \min_{u} \ \lerp{Q}_x(u), \quad x \in \setg{X}.
\end{align*}

\noindent\textit{(i) Difference between $\dpo$ and $\idpo$:} Let 
$u\opt \in \argmin_{u} Q(x,u) \subseteq \setc{U}(x),$
so that $\dpo[J](x) = Q(x,u\opt)$ and $\idpo [J] (x) \leq \lerp{Q}(x,u\opt)$. 
Also, let $z\opt \in \argmin_{z \in \setg{X}} \norm{z - \dyn(x,u\opt)}$. 
Then,
\begin{align*}
\idpo[J](x) - \dpo[J] (x) &\leq \lerp{Q}(x,u\opt) - Q(x,u\opt) \\
&= \lerp{\disc{J}}\big(\dyn(x,u\opt)\big) - \lerp{\disc{J}}(z\opt) + J(z\opt) - J\big(\dyn(x,u\opt)\big),
\end{align*}
where we used the assumption that $\lerp{\disc{J}}(z\opt) = J(z\opt)$ for $z\opt \in \setg{X}$. Hence,
\begin{align*}
\idpo[J](x) - \dpo[J] (x) & \leq  \big[ \lip(J) + \lip(\lerp{\disc{J}}) \big] \cdot \norm{z\opt - \dyn(x,u\opt)}  \\
&=  \big[ \lip(J) + \lip(\lerp{\disc{J}}) \big] \cdot \min_{z \in \setg{X}} \norm{z - \dyn(x,u\opt)} \\
& \leq  \big[ \lip(J) + \lip(\lerp{\disc{J}}) \big] \cdot \max_{z' \in \setc{X}} \ \min_{z \in \setg{X}} \norm{z - z'} \\
&=\big[ \lip(J) + \lip(\lerp{\disc{J}}) \big] \cdot \dish (\setc{X}, \setg{X}) = e_1,
\end{align*}
where for the second inequality we used the fact that $\dyn(x,u\opt) \in \setc{X}$. 
We can use the same line of arguments by defining $\tilde{u}\opt \in \argmin_{u} \lerp{Q}(x,u)$, and $\tilde{z}\opt \in \argmin_{z \in \setg{X}} \norm{z - \dyn(x,\tilde{u}\opt)}$ to show that $\idpo[J](x) - \dpo[J] (x) \leq  e_1$. 
Combining these results, we have
\begin{align}\label{eq:diff idpo and dpo}
-  e_1 \leq \idpo[J](x) - \dpo[J] (x) \leq  e_1.
\end{align}

\noindent\textit{(ii) Difference between $\idpo$ and $\ddpo$:}  First note that, by construction, we have $\idpo[J](x) \leq \ddpo[\disc{J}] (x)$. 
Now, let 
$\tilde{u}\opt \in \argmin_{u} \lerp{Q}(x,u) \subseteq \setc{U}(x),$ 
so that $\idpo [J](x) = \lerp{Q}(x,\tilde{u}\opt)$. 
Also, let 
$\bar{u}\opt \in \argmin_{u \in \setd{U}(x)} \norm{u - \tilde{u}\opt}$, 
and note that $\ddpo [\disc{J}](x) \leq  \lerp{Q}(x,\bar{u}\opt)$. 
Then, using the fact that $\lerp{Q}$ is Lipschitz continuous, we have
\begin{align*}
0 \leq \ddpo [\disc{J}](x) - \idpo [J](x) &\leq \lerp{Q}(x,\bar{u}\opt) - \lerp{Q}(x,\tilde{u}\opt) 
 \leq \lip(\lerp{Q}_x) \cdot \norm{\bar{u}\opt - \tilde{u}\opt} \\
& \leq  \big[ \lip(J) + \lip(C) \big] \cdot \min_{u \in \setd{U}(x)} \norm{u - \tilde{u}\opt} 
\\ &
\leq  \big[ \lip(J) + \lip(C) \big] \cdot \max_{u' \in \setc{U}(x)} \ \min_{u \in \setd{U}(x)} \norm{u - u'} \\
& = \big[ \lip(J) + \lip(C) \big] \cdot \dish \big(\setc{U}(x), \setd{U}(x)\big) = e_2(x),
\end{align*}
Combining this last result with the inequality~\eqref{eq:diff idpo and dpo}, we derive the bounds of the proposition. 
\end{proof}

%============================================
\section{Technical proofs} \label{app:proofs}

\subsection{Proof of Lemma~\ref{lem:conj vs. d-conj}} \label{proof:conj vs. d-conj}

Let $y \in \R^n$, and observe that (recall that $\disc{h} (x) = h(x)$ for all $x \in \setd{X} \subset \setc{X}$)
\begin{equation*}
\lftd{h}(y) = \max_{x \in \setd{X}} \{\inner{y}{x} - \disc{h}(x) \} \leq \max_{x \in \setc{X}} \{\inner{y}{x} - h(x) \} = \lftc{h}(y). 
\end{equation*} 
This settles the first inequality in~\eqref{eq:conj vs. d-conj I} and \eqref{eq:conj vs. d-conj II}. 
Also, observe that if $\partial \lftc{h}(y) = \emptyset$, then the upper bound in~\eqref{eq:conj vs. d-conj I} becomes trivial, i.e., $\lftc{h}(y) = +\infty$, $\lftd{h}(y) < +\infty$, and $\wt{e}_1 = +\infty$. 
Now, assume that $\partial \lftc{h}(y) \neq \emptyset$, and let $x \in \partial \lftc{h}(y)$ so that $h(x) + \lftc{h} (y) = \inner{y}{x}$ \cite[Prop.~5.4.3]{Bertsekas09}. 
Also, let $ \tilde{x} \in \argmin_{z \in \setd{X}} \norm{x-z}$, and note that $\lftd{h}(y) \geq \inner{y}{\tilde{x}} - \disc{h}(\tilde{x})$. 
Then,
\begin{align*}
\lftc{h}(y) - \lftd{h}(y) &\leq \inner{y}{x - \tilde{x}} - h(x) + \disc{h}(\tilde{x}) \\
&\leq \big[ \norm{y} + \lip \big( h; \{x\} \cup \setd{X} \big) \big] \cdot \norm{x - \tilde{x}} \\
&= \big[ \norm{y} + \lip \big( h; \{x\} \cup \setd{X} \big) \big] \cdot \dist(x,\setd{X}) .
\end{align*}
Hence, by minimizing over $x \in \partial \lftc{h}(y)$, we derive the upper bound provided in~\eqref{eq:conj vs. d-conj I}. 
Finally, the additional constraint of compactness of $\setc{X} = \dom(h)$ implies that $\partial  \lftc{h}(y) \cap \setc{X} \neq \emptyset$. 
Hence, we can choose $x \in \partial \lftc{h}(y) \cap \setc{X}$  and use Lipschitz-continuity of $h$ to write
\begin{align*}
\lftc{h}(y) - \lftd{h}(y) 
&\leq  \big[ \norm{y} + \lip \big( h; \{x\} \cup \setd{X} \big) \big] \cdot \dist(x,\setd{X})  \\
&\leq  \big[ \norm{y} + \lip (h) \big] \cdot \max_{z \in \setc{X}} \dist(z,\setd{X}) = \wt{e}_2(y,h,\setd{X}).
\end{align*}

\subsection{Proof of Lemma~\ref{lem:conj via lerp}} \label{proof:conj via lerp}

Let us first consider the case $y \in \co (\setg{Y})$. 
The value of the multi-linear interpolation $\llerp{\lftcd{h}} (y)$ is a convex combination of $\lftcd{h}(y^{(k)}) = \lftc{h}(y^{(k)})$ over the grid points $y^{(k)} \in \setg{Y}, \ k \in {1, \ldots, 2^n}$, located at the vertices of the hyper-rectangular cell that contains $y$  such that 
$$
y = \ssum_{k} \alpha^{(k)} \ y^{(k)} \ \ \ \text{and} \ \ \ \llerp{\lftcd{h}} (y) = \ssum_{k} \alpha^{(k)} \ \lftc{h} (y^{(k)}),
$$
where $\sum_k \alpha^{(k)} = 1$ and $\alpha^{(k)} \in [0,1]$. 
Then,
\begin{align}\label{eq:proof lem conj via lerp 1}
\lftc{h} (y) = \lftc{h} \left(\ssum_{k} \alpha^{(k)} \ y^{(k)} \right) \leq \ssum_{k} \alpha^{(k)} \ \lftc{h} (y^{(k)})  = \llerp{\lftcd{h}} (y),
\end{align}
where the inequality follows from convexity of $\lftc{h}$. 
Also, notice that
\begin{align*}
\llerp{\lftcd{h}} (y) &= \ssum_{k} \alpha^{(k)} \ \lftc{h}(y^{(k)}) = \ssum_{k} \alpha^{(k)} \ \max\limits_{x \in \setc{X}} \left\{ \inner{y^{(k)}}{x} - h (x) \right\} \\
& = \ssum_{k} \alpha^{(k)} \ \max\limits_{x \in \setc{X}} \left\{ \inner{y}{x} - h (x) + \inner{y^{(k)} - y}{x} \right\} \\
& \leq \ssum_{k} \alpha^{(k)} \ \max\limits_{x \in \setc{X}} \left\{ \inner{y}{x} - h (x) +  \norm{y^{(k)} - y} \cdot  \norm{x} \right\} \\
& \leq \ssum_{k} \alpha^{(k)} \ \max\limits_{x \in \setc{X}} \left\{ \inner{y}{x} - h (x) + \diam{\setc{X}} \cdot \dist(y,\setg{Y}) \right\}.
\end{align*} 
Then, using $\sum_k \alpha^k = 1$, we have
\begin{align}\label{eq:proof lem conj via lerp 2}
\llerp{\lftcd{h}} (y) &\leq \max_{x \in \setc{X}} \left\{ \inner{y}{x} - h (x) \right\} + \diam{\setc{X}} \cdot \dist(y,\setg{Y}) = \lftc{h} (y) + \diam{\setc{X}} \cdot \dist(y,\setg{Y}).
\end{align}
Combining the two inequalities~\eqref{eq:proof lem conj via lerp 1} and~\eqref{eq:proof lem conj via lerp 2} gives us the inequality~\eqref{eq:conj via lerp error} in the lemma.

We next consider the case $y \not\in \co (\setg{Y})$ under the extra assumption $\co(\setsg{Y}) \supseteq \setc{L}(h)$. 
Note that this assumption implies that (consult the notation preceding the lemma):
\begin{itemize}
\item  $\setc{L}(h)$ is bounded ($h$ is Lipschitz continuous); and,
\item  $y_i^1 < y_i^2 \leq \lip_i^-(h)$ and $\lip_i^+(h) \leq y_i^{Y_i-1} < y_i^{Y_i}$ for all $i \in \{1,\ldots,n\}$.
\end{itemize}
To simplify the exposition, we consider the two-dimensional case ($n=2$), while noting that the provided arguments can be generalized to higher dimensions.  
So, let $\setg{Y} = \setg{Y}_1 \times \setg{Y}_2$, where $\setg{Y}_i \ (i = 1,2)$ is the finite set of real numbers $y_i^1 < y_i^2 < \ldots < y_i^{Y_i}$ with $Y_i \geq 3$. 
Let us further simplify the argument by letting $y = (y_1,y_2) \not\in\co (\setg{Y}) $ be such that $ y_1 < y_1^{1}$ and $ y_2^{1} \leq y_2 \leq y_2^{2}$, 
so that computing $\llerp{\lftcd{h}}(y)$ involves extrapolation in the first dimension and interpolation in the second dimension; 
see Figure~\ref{fig:grid} for a visualization of this setup. 
Since the extension uses LERP, using the points depicted in Figure~\ref{fig:grid}, we can write
\begin{align}\label{eq:lem conj via lerp 4}
\llerp{\lftcd{h}} (y) = \alpha \ \llerp{\lftcd{h}} (y') + (1-\alpha) \ \llerp{\lftcd{h}} (y''),
\end{align} 
where $\alpha = (y_1^2-y_1)/(y_1^2-y_1^1)$, and (recall that $\lftcd{h}(y) = \lftc{h}(y)$ for $y \in \setg{Y}$)
\begin{align} \label{eq:lem conj via lerp 1}
\begin{array}{l}
\llerp{\lftcd{h}} (y') = \beta \ \lftcd{h} (y^{1,1}) + (1-\beta) \ \lftcd{h} (y^{1,2}) = \beta \ \lftc{h} (y^{1,1}) + (1-\beta) \ \lftc{h} (y^{1,2}), \\
\llerp{\lftcd{h}} (y'') = \beta \ \lftcd{h} (y^{1,2}) + (1-\beta) \ \lftcd{h} (y^{2,2}) = \beta \ \lftc{h} (y^{1,2}) + (1-\beta) \ \lftc{h} (y^{2,2}),
\end{array}
\end{align}
where $\beta = (y_2^2-y_2)/(y_2^2-y_2^1)$. 
In Figure~\ref{fig:grid}, we have also paired each of the points of interest in the dual domain with its corresponding maximizer in the primal domain. 
That is, for $\xi = y,y',y'',y^{1,1},y^{1,2},y^{1,2},y^{2,2}$, we have respectively identified $\eta = x,x',x'',x^{1,1},x^{1,2},x^{1,2},x^{2,2} \in \setc{X}$, where $\xi \in \partial h (\eta)$ so that
\begin{align}\label{eq:lem conj via lerp 3}
\lftc{h}(\xi) = \inner{\eta}{\xi} - h(\eta).
\end{align} 
We now list the \emph{implications} of the assumption $y_1^1 < y_1^2 \leq \lip_1^-(h)$; Figure~\ref{fig:implication} illustrates these implications in the one-dimensional case: 
\begin{itemize}
\item[I.1.] We have $\lftc{h}(y) = \alpha \ \lftc{h} (y') + (1-\alpha) \ \lftc{h} (y'')$.
\item[I.2.] We can choose the maximizers in the primal domain such that
\begin{itemize}
\item[I.2.1.] $x^{1,1} = x^{2,1}$, $x^{1,2} = x^{2,2}$, and $x = x' = x''$;
\item[I.2.2.] $x_1^{1,1} = x_1^{1,2} = x_1 = \min\limits_{(z_1,z_2) \in \setc{X}} z_1$.
\end{itemize}
\end{itemize}

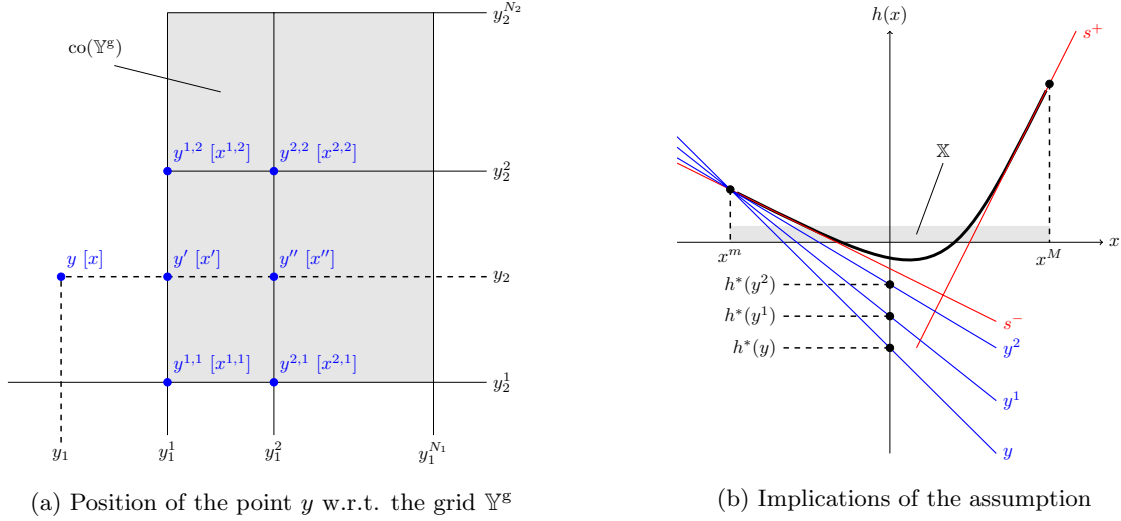
\begin{figure}[t]
\begin{subfigure}{.5\textwidth}
  \centering
  \scalebox{.7}{\begin{tikzpicture}

\filldraw[fill=black!10!white, draw=black!10!white] (2,1) rectangle (7,8);

\path (0,0) node(y1)[below=4pt] {$y_1$}
	  (2,0) node(y11)[below] {$y_1^1$}
	  (4,0) node(y12)[below] {$y_1^2$}
	  (7,0) node(y1n)[below] {$y_1^{N_1}$}
	  (8,1) node(y21)[right] {$y_2^1$}
	  (8,3) node(y2)[right] {$y_2$}
	  (8,5) node(y22)[right] {$y_2^2$}
	  (8,8) node(y2n)[right] {$y_2^{N_2}$};

\draw[thin] 
(y11) -- (2,8) 
(y12) -- (4,8)
(y1n) -- (7,8)
(y21) -- (-1,1) 
(y22) -- (2,5)
(y2n) -- (2,8);

\draw[thick, dashed] 
(y1) -- (0,3) 
(y2) -- (0,3);
	  
\filldraw[blue] 
(0,3) circle (2pt) node[right=13pt, above=1pt] {$y \ [x]$}
(2,3) circle (2pt) node[right=16pt, above=1pt] {$y' \ [x']$}
(4,3) circle (2pt) node[right=18pt, above=1pt] {$y'' \ [x'']$}
(2,1) circle (2pt) node[right=24pt, above=1pt] {$y^{1,1} \ [x^{1,1}]$}
(4,1) circle (2pt) node[right=24pt, above=1pt] {$y^{2,1} \ [x^{2,1}]$}
(2,5) circle (2pt) node[right=24pt, above=1pt] {$y^{1,2} \ [x^{1,2}]$}
(4,5) circle (2pt) node[right=24pt, above=1pt] {$y^{2,2} \ [x^{2,2}]$};

\draw[ultra thin] (3,6.5) -- (1,7) node[left=10pt, above=1pt] {$\co(\setg{Y})$};

\end{tikzpicture}}
  \caption{Position of the point $y$ w.r.t. the grid $\setg{Y}$ }
  \label{fig:grid}
\end{subfigure}%
\begin{subfigure}{.5\textwidth}
  \centering
  \scalebox{.7}{\begin{tikzpicture}

\filldraw[fill=black!10!white, draw=black!10!white] (-3,0) rectangle (3,.3);

\path (-3,1) node(a) {}
	  (3,3) node(b) {};
	  
\draw[->,very thin] (-4,0) -- (4,0) node[anchor=west] {$x$};
\draw[->,very thin] (0,-4) -- (0,4) node[anchor=south] {$h(x)$};

\draw[thick, dashed] 
(-3,1) -- (-3,0) node[anchor=north] {$x^m$}
(3,3) -- (3,0) node[anchor=north] {$x^M$};

\draw[ultra thick] (a) .. controls (1,-1) .. (b);

\draw[semithick, red] 
(-4,1.5) -- (2,-1.5) node[anchor=west] {$s^-$}
(0.5,-2) -- (3.5,4) node[anchor=west] {$s^+$};

\draw[semithick, blue] 
(-4,1.6) -- (2,-2) node[anchor=west] {$y^2$}
(-4,1.8) -- (2,-3) node[anchor=west] {$y^1$}
(-4,2) -- (2,-4) node[anchor=west] {$y$};

\draw[thick, dashed] 
(0,-0.8) -- (-2,-0.8) node[anchor=east] {$\lftc{h}(y^2)$}
(0,-1.4) -- (-2,-1.4) node[anchor=east] {$\lftc{h}(y^1)$}
(0,-2) -- (-2,-2) node[anchor=east] {$\lftc{h}(y)$};

\draw[ultra thin] (0.5,0.15) -- (1,1.5) node[anchor=south] {$\setc{X}$};

\filldraw 
(a) circle (2pt)
(b) circle (2pt)
(0,-0.8) circle (2pt)
(0,-1.4) circle (2pt)
(0,-2) circle (2pt);

\end{tikzpicture}}
  \caption{Implications of the assumption}
  \label{fig:implication}
\end{subfigure}
\caption{Illustration of the proof of Lemma~\ref{lem:conj via lerp}. 
(a)~The dual grid $\setg{Y}$ and the position of the point $y$ w.r.t. the grid. 
The blue dots show the points of interest and their corresponding maximizer in the primal domain. E.g., ``$y \ [x]$'' implies that $y \in \partial h (x)$, where $x\in \setc{X}$, so that $\inner{x}{y} = h(x) + \lftc{h}(y)$. 
(b)~Illustration of the implications of the assumption $y^1 < y^2 \leq s^- = \lip^-(h)$ in the one-dimensional case. 
The colored (red and blue) variables denote the slope of the corresponding lines. 
Note that $\{y , y^1,y^2\} \subset \partial h (x^m)$, where $x^m = \min_{x \in \setc{X}} x$. 
Indeed, for all $y \leq s^-$, the conjugate $\lftc{h}(y) = \inner{x^m}{y} - h(x^m)$ is a linear function with slope $x^m$. 
In particular, for $y < y^1$, we have $\lftc{h}(y) = \alpha \lftc{h}(y^1) + (1-\alpha) \lftc{h}(y^2)$, where $\alpha = (y^2-y)/(y^2-y^1)$.  }
\label{fig:proof of conj via lerp}
\end{figure}

With these preparatory discussions, we can now consider the error of extrapolative discrete conjugation at the point $y$. 
In this regard, first note that $\{y', y''\} \subset \co(\setg{Y})$, and hence we can use the result of the first part of the lemma to write
\begin{eqnarray} \label{eq:lem conj via lerp 2}
\begin{array}{l} 
\llerp{\lftcd{h}} (y') = \lftc{h}(y') + e', \quad
\llerp{\lftcd{h}} (y'') = \lftc{h}(y'') + e'',
\end{array}
\end{eqnarray} 
where $\{e',e'' \} \subset [0, \diam{\setc{X}} \cdot \dish(\{y', y''\},\setg{Y})]$. 
We claim that these error terms are equal. 
Indeed, from \eqref{eq:lem conj via lerp 1} and \eqref{eq:lem conj via lerp 2}, we have
\begin{align*}
e' - e'' = \beta  \left[ \lftc{h} (y^{1,1}) - \lftc{h} (y^{2,1}) \right]+ (1-\beta) \left[ \lftc{h} (y^{1,2}) - \lftc{h} (y^{2,2}) \right] + \lftc{h}(y'') - \lftc{h}(y').
\end{align*}
Then, using the pairings in~\eqref{eq:lem conj via lerp 3} and the implication I.2, we can write 
\begin{align*}
e' - e'' &\overset{(I.2.1)}{=} \beta  \inner{x^{1,1}}{y^{1,1}-y^{2,1}}+ (1-\beta) \inner{x^{1,2}}{y^{1,2}-y^{2,2}} + \inner{x}{y''-y'} \\
& \ \ = \ \ \beta  \inner{x^{1,1}}{(y_1^1- y_1^2, 0)}+ (1-\beta) \inner{x^{1,2}}{(y_1^1- y_1^2, 0)} + \inner{x}{(y_1^2 - y_1^1 , 0)} \\
& \ \ = \ \  \left( \beta x_1^{1,1} + (1-\beta)x_1^{1,2} - x_1 \right) (y_1^1- y_1^2) \overset{(I.2.2)}{=} 0.
\end{align*}
With this result at hand, we can employ the equality~\eqref{eq:lem conj via lerp 4} and the implication I.1 to write
\begin{align*}
\llerp{\lftcd{h}} (y) - \lftc{h}(y) &= \alpha \left[ \llerp{\lftcd{h}} (y') - \lftc{h}(y') \right] + (1-\alpha) \left[ \llerp{\lftcd{h}} (y'') - \lftc{h}(y'') \right] = \alpha e' + (1-\alpha) e'' = e'.
\end{align*}
That is, 
\begin{align*}
0\leq \llerp{\lftcd{h}} (y) - \lftc{h}(y) \leq \diam{\setc{X}} \cdot \dish(\{y', y''\},\setg{Y}) \leq \diam{\setc{X}} \cdot \dish\big(\co (\setg{Y}),\setg{Y}\big),
\end{align*}
where for the last inequality we used the fact that $\{y', y''\} \subset \co(\setg{Y})$.

\subsection{Proof of Corollary~\ref{cor:disc conj via lerp}} \label{proof:disc conj via lerp}

The first statement immediately follows from Lemma~\ref{lem:conj via lerp} since the finite set $\setd{X}$ is compact. 
For the second statement, the extra condition $\co(\setsg{Y}) \supseteq \setc{L}(h)$ has the same implications as the ones provided in the proof of Lemma~\ref{lem:conj via lerp} in Appendix~\ref{proof:conj via lerp}. 
Hence, following the same arguments, we can show that provided bounds hold for all $y \in \R^n$ under the given condition.

\subsection{Proof of Lemma~\ref{lem:CDP op}} \label{proof:CDP op}

Using the definition of conjugate transform, we have
\begin{align*}
\cdpo [J](x) &= \max_{y \in \R^n} \ \min_{u, z \in \R^n} \left\{\cost(x,u) + J (z) + \inner{y}{\dynx(x) + \dynu(x) u-z}  \right\} \\
&= \max_{y} \left\{ \inner{y}{\dynx(x)} - \max_{u} \left[ \inner{-\dynu(x)\tr y}{u} - \cost(x,u) \right] - \max_{ z} \left[ \inner{y}{z} -J (z) \right]  \right\} \\
&= \max_{y } \left\{ \inner{y}{\dynx(x)} - \lftc{\cost_x}(-\dynu(x)\tr y) -\lftc{J} (y)  \right\} \\
&= \max_{y } \left\{ \inner{y}{\dynx(x)} - \phi_x(y)  \right\} 
= \lftc{\phi_x}\big(\dynx(x)\big).
\end{align*}

\subsection{Proof of Theorem~\ref{thm:complexity d-CDP Alg 1}} \label{proof:complexity d-CDP Alg 1}

In what follows, we provide the time complexity of each line of Algorithm~\ref{alg:d-CDP general}.  
The LLT of line~\ref{line_alg_g:LLT of J} requires $\ord (X+Y)$ operations; see Remark~\ref{rem:complexity of LLT}. 
By Assumption~\ref{As:conj cost func alg 1}, computing $\disc{\psi}_x$ in line~\ref{line_alg_g:h} has a complexity of $\ord (Y)$. 
The minimization via enumeration in line~\ref{line_alg_g:output} also has a complexity of $\ord (Y)$. 
This, in turn, implies that the \texttt{for loop} over~$x \in \setg{X}$ requires $\ord (XY)$ operations. 
Hence, the total time complexity of $\ord (XY)$. 

\subsection{Proof of Proposition~\ref{prop:d-CDP op}} \label{proof:d-CDP op}

We can use the representation~\eqref{eq:d-CDP op} and the definition~\eqref{eq:conj cost func alg 1} to obtain
\begin{align*}
\dcdpo [\disc{J}](x) &= \max_{y \in \setg{Y}} \ \{ \inner{\dynx(x)}{y} - \disc{\psi}_x(y) \} \\
& = \max_{y \in \setg{Y}} \ \left\{ \inner{\dynx(x)}{y} - \lftc{\cost_x}(-\dynu(x)\tr y) - \lftdd{J}(y) \right\} \\ 
& = \max_{y \in \setg{Y}} \ \left\{ \inner{\dynx(x)}{y} - \max_{u \in \dom\cost(x,\cdot)} \left[ \inner{-\dynu(x)\tr y}{u} - \cost(x,u) \right] - \lftdd{J}(y) \right\} \\ 
& = \max_{y \in \setg{Y}} \  \min_{u \in \dom\cost(x,\cdot)} \left\{ \cost(x,u) +  \inner{y}{\dyn (x,u)} - \lftdd{J}(y)  \right\},
\end{align*} 
Since $C$ is convex in $u$ and the mapping $\dyn$ is affine in $u$, the objective function of this maximin problem is convex in $u$, with $\dom \big(C(x,\cdot)\big)$ being compact. 
Also, the objective function is Ky Fan concave in $y$, which follows from the convexity of $\lftd{J}$. 
Then, by the Ky Fan's Minimax Theorem (see, e.g., \cite[Thm.~A]{Joo82}), we can swap the maximization and minimization operators to obtain
\begin{align*}
\dcdpo [\disc{J}](x) & =  \min_{u \in \dom C(x,\cdot)} \  \max_{y \in \setg{Y}} \  \left\{ \cost(x,u) +  \inner{y}{\dyn(x,u)} - \lftdd{J}(y)  \right\} \nonumber \\
& = \min_{u } \left\{ \cost(x,u) + \bcdd{J}\big( \dyn(x,u) \big) \right\}.
\end{align*}

\subsection{Proof of Theorem~\ref{thm:error d-CDP Alg 1}} \label{proof:error d-CDP Alg 1}

Fix $x \in \setg{X}$ and observe that
\begin{align}
\dpo [J](x) - \dcdpo [\disc{J}](x) = \left[ \dpo [J](x) - \cdpo [J](x) \right] + \left[ \cdpo [J](x) - \dcdpo [\disc{J}](x) \right]. \label{eq:error d-CDP alg 1 - (0)}
\end{align}
Let us first note that the convexity $\cost :\setc{X} \times \setc{U} \ra \Ru $ (in $u$) and $J:\setc{X} \ra \R$ implies that the duality gap $\cdpo [J]- \dcdpo [\disc{J}]$ in \eqref{eq:error d-CDP alg 1 - (0)} is zero. 
Indeed, following a similar argument as the one provided in the proof of Proposition~\ref{prop:d-CDP op} in Appendix~\ref{proof:d-CDP op}, and using Sion's Minimax Theorem (see, e.g., \cite[Thm.~3]{Simons95}), we can show that 
$$
\cdpo [J](x) = \min_{u} \left\{ \cost(x,u) + \bcc{J} \big( \dyn(x,u) \big) \right\}, \quad x \in \setc{X}.
$$
Then, since $J$ is a proper, closed, convex function, we have $\bcc{J} = J$, and hence $\cdpo [J] = \dpo [J]$. 
We next consider the discretization error $\cdpo [J]- \dcdpo [\disc{J}]$ in \eqref{eq:error d-CDP alg 1 - (0)}. 
From \eqref{eq:CDP op conj a} and \eqref{eq:d-CDP op a}, we have 
\begin{align}
\cdpo [J](x) - \dcdpo [\disc{J}](x) &= \lftc{\phi_x} \big(\dynx(x) \big) - \lftd{\psi_x} \big(\dynx(x) \big) \nonumber \\
&= \left[ \lftc{\phi_x} \big(\dynx(x) \big) - \lftd{\phi_x} \big(\dynx(x) \big) \right] + \left[ \lftd{\phi_x} \big(\dynx(x) \big) - \lftd{\psi_x} \big(\dynx(x) \big) \right], \label{eq:error d-CDP alg 1 - (01)}
\end{align}
where $\disc{\phi_x}:\setg{Y} \ra \R $ is the discretization of $\phi_x:\R^n \ra \R$. 
For $\lftc{\phi_x} - \lftd{\phi_x} $ in \eqref{eq:error d-CDP alg 1 - (01)}, 
by Lemma~\ref{lem:conj vs. d-conj}, we have
\begin{align*}
0 \leq \lftc{\phi_x} \big(\dynx(x) \big) - \lftd{\phi_x} \big(\dynx(x) \big) &\leq \wt{e}_1(\dynx(x), \phi_x, \setg{Y}) \\
&= \min\limits_{y \in \partial  \lftc{\phi_x}(\dynx(x) )} \bigg\{ \big[ \norm{\dynx(x)} + \lip \big( \phi_x; \{y\} \cup \setg{Y} \big) \big] \cdot \dist(y,\setg{Y}) \bigg\} \\ 
&\leq \min\limits_{y \in \partial  \dpo [J] (x)} \bigg\{ \big[ \norm{\dynx(x)} + \norm{\dynu(x)} \cdot \diam{\setc{U}} + \diam{\setc{X}} \big] \cdot \dist(y,\setg{Y}) \bigg\},
\end{align*}
where we used the fact that $\lftc{\phi_x} \big(\dynx(\cdot)\big) = \cdpo [J] (\cdot) = \dpo [J] (\cdot)$, and 
\begin{align*}
\lip \big( \phi_x(\cdot) \big) &\leq \lip \big( \lftc{C_x} (-\dynu(x)\tr \cdot) \big) + \lip \big( \lftc{J} (\cdot) \big) \\
& \leq  \norm{\dynu(x)} \cdot \lip( \lftc{C_x} ) + \lip ( \lftc{J} ) \\
&\leq \norm{\dynu(x)} \cdot \diam{\dom(C(x,\cdot))} +  \diam{\dom(J)} \\
& \leq \norm{\dynu(x)} \cdot \diam{\setc{U}} + \diam{\setc{X}}.
\end{align*}
Hence, 
\begin{align} \label{eq:error d-CDP alg 1 - (1)}
0 \leq \lftc{\phi_x} \big(\dynx(x) \big) - \lftd{\phi_x} \big(\dynx(x) \big) &\leq  \big[ \norm{\dynx(x)} + \norm{\dynu(x)} \cdot \diam{\setc{U}} + \diam{\setc{X}} \big] \cdot \min\limits_{y \in \partial  \dpo [J] (x)} \dist(y,\setg{Y}) \nonumber \\
&= \big[ \norm{\dynx(x)} + \norm{\dynu(x)} \cdot \diam{\setc{U}} + \diam{\setc{X}} \big] \cdot \dist\big(\partial  \dpo [J] (x),\setg{Y}\big) = e_1(x)
\end{align}
For $\lftd{\phi_x} - \lftd{\psi_x}$ in \eqref{eq:error d-CDP alg 1 - (01)}, first observe that for each $y \in \setg{Y}$, we have (see \eqref{eq:CDP op conj b} and \eqref{eq:d-CDP op b}, and recall that $\disc{h}$ is simply a sampled version of $h$)
\begin{align*}
\disc{\phi}_x(y) -  \disc{\psi}_x(y) = \lftcd{J}(y) - \lftdd{J}(y) = \lftc{J}(y) - \lftd{J}(y).
\end{align*}
Moreover, we can use Lemma~\ref{lem:conj vs. d-conj}, and the fact that $\dom(J) = \setc{X}$ is compact, to write 
\begin{align*}
0 \leq \lftc{J}(y) - \lftd{J}(y) & \leq  \left[ \norm{y} +  \lip(J) \right] \cdot \dish (\setc{X}, \setg{X}) \\
%& \leq \max_{y \in \setg{Y}}  \left[ \norm{y} +  \lip(J) \right] \cdot \dish (\setc{X}, \setg{X}) \\
& \leq  \left[ \diam{\setg{Y}} +  \lip(J) \right] \cdot \dish (\setc{X}, \setg{X})  = e_2.
\end{align*}
That is, 
$$0 \leq \disc{\phi}_x(y) -  \disc{\psi}_x(y) \leq e_2, \quad \forall y \in \setg{Y}. $$ 
Then, using the definition of discrete conjugate, we have
\begin{align*} 
0 \leq \lftd{\psi_x} \big(\dynx(x) \big) - \lftd{\phi_x} \big(\dynx(x) \big) \leq  e_2.
\end{align*}
Combining the last inequality with the inequality~\eqref{eq:error d-CDP alg 1 - (1)} completes the proof.  

\subsection{Proof of Theorem~\ref{thm:complexity d-CDP Alg 2}} \label{proof:complexity d-CDP Alg 2}

In what follows, we provide the time complexity of each line of Algorithm~\ref{alg:d-CDP separ}. 
The LLT of line~\ref{line_alg_s:LLT of J} requires $\ord (X+Y)$ operations; see Remark~\ref{rem:complexity of LLT}. 
By Assumption~\ref{As:conj cost func alg 2}, computing $\disc{\psi}$ in line~\ref{line_alg_s:h} has a complexity of $\ord (Y)$. 
The LLT of line~\ref{line_alg_s:LLT of h} requires $\ord (Y+Z)$ operations. 
The approximation of line~\ref{line_alg_s:LERP of h} using LERP has a complexity of $\ord (\log Z)$; see Remark~\ref{rem:complexity of LERP}. 
Hence, the \texttt{for loop} over~$x \in \setg{X}$ requires $\ord (X \log Z) = \wt{\ord}(X)$ operations. 
The time complexity of the whole algorithm can then be computed by adding all the aforementioned complexities. 

\subsection{Proof of Theorem~\ref{thm:error d-CDP Alg 2}} \label{proof:error d-CDP Alg 2}

Let $\dcdpo$ denote the output of the implementation of the d-CDP operator~\eqref{eq:d-CDP op separ} via Algorithm~\ref{alg:d-CDP general}. 
Note that the computation of the modified d-CDP operator~$\mdcdpo$~\eqref{eq:d-CDP op separ modified} via Algorithm~\ref{alg:d-CDP separ} differs from that of the d-CDP operator~$\dcdpo$~\eqref{eq:d-CDP op separ} via Algorithm~\ref{alg:d-CDP general} only in the last step. 
To see this, note that $\dcdpo$ \emph{exactly} computes~
$\lftd{\psi}\big( \dynx(x) \big)$ for~$x \in \setg{X}$ (see line~\ref{line_alg_g:output} of Algorithm~\ref{alg:d-CDP general}). 
However, in $\mdcdpo$, the \emph{approximation}~$\llerp{\lftdd{\psi}} \big( \dynx(x) \big)$ is used (see line~\ref{line_alg_s:LERP of h} of Algorithm~\ref{alg:d-CDP separ}), where the approximation uses LERP over the data points~$\lftdd{\psi}:\setg{Z} \ra \R$. 
By Corollary~\ref{cor:disc conj via lerp} and the assumption $\co (\setg{Z}) \supseteq \dynx(\setg{X})$, this leads to an over-approximation of~$\lftd{\psi}$, with the upper bound 
$$
e_3 = \diam{\setg{Y}} \cdot \max_{x \in \setg{X}} \dist\big(\dynx (x), \setg{Z} \big) = \diam{\setg{Y}} \cdot \dish\big(\dynx (\setg{X}), \setg{Z} \big).
$$ 
Hence, compared to $\dcdpo$, the operator $\mdcdpo$ is an over-approximation with the difference bounded by $e_3$, i.e.,  
\begin{equation}\label{eq:error d-CDP alg 2 - (1)}
0 \leq \mdcdpo [\disc{J}] (x) - \dcdpo [\disc{J}] (x) \leq  e_3, \quad \forall x \in \setg{X}.
\end{equation}
The result then follows from Theorem~\ref{thm:error d-CDP Alg 1}. 
Indeed, using the definition of $\dcdpo$~\eqref{eq:d-CDP op separ}, we can define
\begin{align*}
& \disc{\psi}(y) \Let \lftc{\costu}(-B\tr y) + \lftdd{J}(y), & y \in \setg{Y}, \\
& \idcdpo [\disc{J}](x) \Let \dcdpo [\disc{J}](x) - \costx(x) = \lftd{\psi} \big( \dynx(x) \big), &  x \in \setg{X}.
\end{align*}
Similarly, using the DP operator~\eqref{eq:DP op separ}, we can also define 
$$\iidpo [J](x) \Let \dpo [J](x)- \costx(x) = \min_{u}\left\{ \costu(u) + J \big(\dyn(x,u)\big) \right\}.$$
Then, by Theorem~\ref{thm:error d-CDP Alg 1}, for all $x \in \setg{X}$, it holds that  
\begin{align} \label{eq:error d-CDP alg 2 - (2)}
\begin{array}{ll}
& - e_2 \leq \iidpo [J](x) - \idcdpo [\disc{J}](x) = \dpo [J](x) - \dcdpo [\disc{J}](x) \leq e^m_1(x),
\end{array}
\end{align}
where $e_2$ is given in \eqref{eq:err terms d-CDP Alg 1}, and 
\begin{align*}
e^m_1(x) &= \big[ \norm{\dynx(x)} + \norm{B} \cdot\diam{\setc{U}} + \diam{\setc{X}} \big] \cdot \dist\big(\partial  \iidpo [J] (x),\setg{Y}\big) \\
&=  \big[ \norm{\dynx(x)} + \norm{B} \cdot \diam{\setc{U}} + \diam{\setc{X}} \big] \cdot \dist\big(\partial  \big(\dpo [J] - \costx \big) (x),\setg{Y}\big).
\end{align*}
Combining the inequalities \eqref{eq:error d-CDP alg 2 - (1)} and \eqref{eq:error d-CDP alg 2 - (2)} completes the proof.

%====================================

\section{Extended algorithms \& further numerical examples} \label{app:num example} 

In this section, we consider the extensions of the proposed d-CDP algorithm and their implications on its complexity. 
In particular, the extension to stochastic systems with additive disturbance and the possibility of numerical computation of the conjugate of the (input-dependent) stage cost are discussed. 
The pseudo-codes for the multistep implementation of the extended d-CDP algorithms are provided in Algorithms~\ref{alg:d-CDP general extended} and~\ref{alg:d-CDP separ extended}. 
Moreover, we showcase the application of these algorithms in solving the optimal control problem for a simple epidemic model and a noisy inverted pendulum.

\subsection{Extensions of d-CDP algorithm}\label{subsec:extension} 

\subsubsection{Stochastic systems} \label{subsec:extension stochastic}
Consider the stochastic version of the dynamics~\eqref{eq:dyn} described by
\begin{equation*} \label{eq:dyn stoch}
x_{t+1} = \dyn(x_t,u_t) + w_t,
\end{equation*}
where $w_t, \ t =0,\ldots,T-1$, are independent, \emph{additive} disturbances. 
Then, the stochastic version of the CDP operator $\cdpo$ \eqref{eq:CDP op conj} still reads the same, except it takes $\ef{J} (\cdot) \Let \EE_w J(\cdot+w)$ as the input, where $\EE_w$ is the expectation operator w.r.t. $w$. 
In other words, we need to first pass the cost-to-go~$J$ through the ``expectation filter'', and then feed it to the CDP operator. 
The extension of the d-CDP algorithms for handling this type of stochasticity involves similar considerations as we explain next. 

Let us first consider the extension of the d-CDP Algorithm~\ref{alg:d-CDP general} for stochastic dynamics with additive disturbance. 
For illustration, assume that the disturbances are i.i.d. and belong to a finite set $\setd{W} \subset \R^n$, with a known probability mass function (p.m.f.)~$p:  \setd{W} \ra [0,1]$.\footnote{The set $\setd{W}$ can indeed be considered as a discretization of a bounded set of disturbances. 
Of course, one can modify the algorithm by incorporating other schemes for computing/approximating the expectation operation.} 
The corresponding extension then involves applying $\dcdpo$~\eqref{eq:d-CDP op} to $\disc{\ef{J}}(x):\setg{X} \ra \Ru$ given by 
\begin{equation}\label{eq:expectation op}
\disc{\ef{J}}(x) = \ssum_{w \in \setd{W}} p(w) \cdot \lerp{\disc{J}}(x+w),
\end{equation}
where $\lerp{[\cdot]}$ is an extension operator (see also line~\ref{line_alg_ge:expect} of Algorithm~\ref{alg:d-CDP general extended}). 
Assuming that a single evaluation of the employed extension operator in \eqref{eq:expectation op} requires $\ord(E)$ operations, 
the stochastic version of the d-CDP Algorithm~\ref{alg:d-CDP general} that utilizes the scheme described above requires $\ord\big(X(WE+Y)\big)$ operations ($\ord(XWE)$ for computing $\disc{\ef{J}}$ and $\ord(XY)$ for applying $\dcdpo$). 
The same extension can be applied to the modified d-CDP operator~$\mdcdpo$~\eqref{eq:d-CDP op separ modified}, as it is done in line~\ref{line_alg_se:expect} of Algorithm~\ref{alg:d-CDP separ extended}. 
In particular, the stochastic version of the modified d-CDP Algorithm~\ref{alg:d-CDP separ} that uses this scheme requires $\wt{\ord}(XWE+Y+Z)$ operations in each iteration. 
On the other hand, the stochastic version of the d-DP operation, described by
\begin{equation}\label{eq:d-DP stoch}
 \ddpo_{\text{s}} [\disc{J}](x) \Let \min_{u \in \setd{U}} \left\{ \cost(x,u) + \EE_w \left[ \lerp{\disc{J}}\big(\dyn(x,u)+w\big) \right]\right\}, \quad x \in \setg{X},
\end{equation} 
has a time complexity of $\ord(XUWE)$. 

%==================================================================================
\begin{algorithm}[b]
\begin{small}
   \caption{Multistep implementation of the extended d-CDP Algorithm~\ref{alg:d-CDP general}}
   \label{alg:d-CDP general extended}
\begin{algorithmic}[1]
	\REQUIRE dynamics~$\dynx: \R^n \ra \R^n, \ \dynu : \R^n \ra \R^{n\times m}$; \\
	discrete stage cost~$\disc{\cost}(x,\cdot): \setg{U} \ra \Ru$ for $x \in \setg{X}$; \\
	discrete terminal cost~$\disc{\cost}_T: \setg{X} \ra \R$; \\
	discrete disturbance~$\setd{W}$ and its p.m.f.~$p: \setd{W} \ra [0,1]$.
	\ENSURE discrete costs-to-go~$\disc{J}_t: \setg{X} \ra \R, \ t=0,1,\ldots, T$.

  	\FOR{each $x \in \setg{X}$} 
    
 		\STATE construct the grid $\setg{V}(x)$; \label{line_alg_ge:const V}

  		\STATE use LLT to compute $\lftdd{\cost_x}: \setg{V}(x) \ra \R$ from $\disc{\cost}(x,\cdot): \setg{U} \ra \Ru$; \label{line_alg_ge:LLT of g}
  		
  	\ENDFOR
  	
  	\STATE $\disc{J}_T(x) \gets \disc{\cost}_T(x)$ for $x \in \setg{X}$;
  	  	
  	\FOR{$t= T, \ldots, 1$} \label{line_alg_ge:for loop over t}

  		\STATE  $\disc{J}_{\mathrm{w},t}(x) \gets \sum_{w \in \setd{W}} p(w) \cdot \lerp{\disc{J_t}}(x+w)$ for $x \in \setg{X}$; \label{line_alg_ge:expect}

  		\STATE construct the grid $\setg{Y}$; \label{line_alg_ge:const Y}
    
    	\STATE use LLT to compute $\lftdd{J}_{\mathrm{w},t}: \setg{Y} \ra \R$ from $\disc{J}_{\mathrm{w},t}: \setg{X} \ra \R$; \label{line_alg_ge:LLT of J}

    	\FOR{each $x \in \setg{X}$} \label{line_alg_ge:for loop over x}

   			\FOR{each $y \in \setg{Y}$}
   				
   				\STATE use LERP to compute $\llerp{\lftdd{\cost_x}}(-\dynu(x)\tr y)$ from $\lftdd{\cost_x}: \setg{V}(x) \ra \R$; \label{line_alg_ge:before h}
   			
    			\STATE  $\disc{\psi}_x(y) \gets \llerp{\lftdd{\cost_x}}(-\dynu(x)\tr y) + \lftdd{J}_{\mathrm{w},t}(y)$; \label{line_alg_ge:h}
    		
  				\ENDFOR
  				
  				\STATE $\disc{J}_{t-1}(x) \gets \max\limits_{y \in \setg{Y}}\{\inner{\dynx(x)}{y} - \disc{\psi}_x(y) \}$. \label{line_alg_ge:output}

  		\ENDFOR

  	\ENDFOR

\end{algorithmic}
\end{small}
\end{algorithm}

\begin{algorithm}
\begin{small}
   \caption{Multistep implementation of the extended d-CDP Algorithm~\ref{alg:d-CDP separ}}
   \label{alg:d-CDP separ extended}
\begin{algorithmic}[1]
	\REQUIRE dynamics~$\dynx: \R^n \ra \R^n, \ B \in \R^{n\times m}$; \\
	discrete state cost~$\disc{\costx}: \setg{X}  \ra \R$; \\
	discrete input stage cost~$\disc{\costu}: \setg{U} \ra \R$; \\
	discrete terminal cost~$\disc{\cost}_T: \setg{X} \ra \R$; \\
	discrete disturbance~$\setd{W}$ and its p.m.f.~$p: \setd{W} \ra [0,1]$.
	\ENSURE discrete costs-to-go~$\disc{J}_t: \setg{X} \ra \R, \ t=0,1,\ldots, T$.
  	
    \STATE construct the grid $\setg{V}$; \label{line_alg_se:const V}

  	\STATE use LLT to compute $\lftdd{\costu}: \setg{V} \ra \R$ from $\disc{\costu}: \setg{U} \ra \R$; \label{line_alg_se:LLT of g}
  	
  	\STATE construct the grid $\setg{Z}$;
  	\STATE $\disc{J}_T(x) \gets \disc{\cost}_T(x)$ for $x \in \setg{X}$;
  	  	
  	\FOR{$t= T, \ldots, 1$} \label{line_alg_se:for loop over t}
  	
  		\STATE  $\disc{J}_{\mathrm{w},t}(x) \gets \sum_{w \in \setd{W}} p(w) \cdot \lerp{\disc{J_t}}(x+w)$ for $x \in \setg{X}$;\label{line_alg_se:expect}
  	
  		\STATE construct the grid $\setg{Y}$; \label{line_alg_se:const Y}
    
    	\STATE use LLT to compute $\lftdd{J}_{\mathrm{w},t}: \setg{Y} \ra \R$ from $\disc{J}_{\mathrm{w},t}: \setg{X} \ra \R$; \label{line_alg_se:LLT of J}
    	
    	\FOR{each $y \in \setg{Y}$}
   				
   			\STATE use LERP to compute $\llerp{\lftdd{\costu}}(-B\tr y)$ from $\lftdd{\costu}: \setg{V} \ra \R$; \label{line_alg_se:before h}
   			
    		\STATE  $\disc{\psi}(y) \gets \llerp{\lftdd{\costu}}(-B\tr y) + \lftdd{J}_{\mathrm{w},t}(y)$; \label{line_alg_se:h}
    		
		\ENDFOR
  
    	\STATE use LLT to compute $\lftdd{\psi}: \setg{Z} \ra \R$ from  $\disc{\psi}: \setg{Y} \ra \R$; \label{line_alg_se:LLT of h}
    
    	\FOR{each $x \in \setg{X}$}
    
 			\STATE  use LERP to compute $\llerp{\lftdd{\psi}}\big( \dynx(x) \big)$ from $\lftdd{\psi}: \setg{Z} \ra \R$; \label{line_alg_se:LERP of h}
     	
    		\STATE $\disc{J}_{t-1}(x) \gets \disc{\costx} (x) + \llerp{\lftdd{\psi}}\big( \dynx(x) \big)$;
    	
   		\ENDFOR

  	\ENDFOR

\end{algorithmic}
\end{small}
\end{algorithm} 
%==================================================================================

\subsubsection{Numerical computation of $\lftc{\cost_x}$ and $\lftc{\costu}$} \label{subsec:extension num conj}
Assumptions~\ref{As:conj cost func alg 1} and \ref{As:conj cost func alg 2} on the availability of the conjugate of the (input-dependent) stage cost can be restrictive. 
Alternatively, we can use \emph{approximate discrete conjugation} for computing these objects numerically. 
Let us begin with describing such a scheme for numerical computation of $\lftc{\cost_x}$ in the d-CDP operator~$\dcdpo$~\eqref{eq:d-CDP op}. 
The scheme has two main steps (see also lines~\ref{line_alg_ge:const V}-\ref{line_alg_ge:LLT of g} and \ref{line_alg_ge:before h}-\ref{line_alg_ge:h} of Algorithm~\ref{alg:d-CDP general extended}):
\begin{itemize} 
\item \textbf{Step 1.} For each $x \in \setg{X}$: 
\begin{itemize}
\item[1.a.] compute/evaluate $\disc{\cost}_x = \disc{\cost}(x,\cdot): \setg{U} \ra \Ru$, where $\setg{U}$ is a \emph{grid-like} discretization of $\setc{U}$;
\item[1.b.] construct the dual grid $\setg{V}(x)$ using the method described below; and,
\item[1.c.] apply LLT to compute $\lftdd{\cost_x}: \setg{V}(x) \ra \R$ using the data points $\disc{\cost}_x: \setg{U} \ra \Ru$. 
\end{itemize}
\item \textbf{Step 2.} For each $y \in \setg{Y}$: use LERP to compute $\llerp{\lftdd{\cost_x}}(-\dynu(x)\tr y)$ from the data points $\lftdd{\cost_x}: \setg{V}(x) \ra \R$, and use the result in \eqref{eq:d-CDP op b} as an approximation of $\lftc{\cost_x}(-\dynu(x)\tr y)$.
\end{itemize} 

This scheme introduces some error that mainly depends on the grids $\setg{U}$ and $\setg{V}(x)$ used for the discretization of the input space and its dual domain, respectively. 
Indeed, we can use Lemmas~\ref{lem:conj vs. d-conj} and Corollary~\ref{cor:disc conj via lerp} to bound  this error. 
We now use those results to provide some guidelines on the construction of the dual grids $\setg{V}(x)$ for each $x \in \setg{X}$. 
By Corollary~\ref{cor:disc conj via lerp}, we can either construct $\setg{V}(x)$ \emph{dynamically} such that $\co \big(\setg{V}(x)\big) \supseteq -\dynu(x)\tr \setg{Y}$ at each iteration, 
or construct a \emph{fixed} grid $\setg{V}(x)$ such that $\co\big(\setsg{V} (x)\big) \supseteq \setc{L}(\disc{\cost}_x) =  \Pi_{i=1}^{m} \left[\lip_i^-(\disc{\cost}_x), \lip_i^+(\disc{\cost}_x)\right]$. 
The former requires $\ord(XY)$ operations \emph{per iteration}, while the latter has a \emph{one-time} computational cost of $\ord (X)$ assuming we have access to $\setc{L}(\disc{\cost}_x)$ for each $x \in \setg{X}$ (see also Remark~\ref{rem:constr V}). 
For this reason, we use the second method.
Then, the problem reduces to computing the ``range of slopes'' of $\disc{\cost}_x$. 
In particular, we can use $\lip_i^- (\disc{\cost}_x) = \min_{u \in \setc{U}} \frac{\partial C(x,u)}{\partial u_i}$ and $\lip_i^+ (\disc{\cost}_x) = \max_{u \in \setc{U}} \frac{\partial C(x,u)}{\partial u_i}$ for each dimension $i =1,2,\ldots,m$.\footnote{If the required maximum and minimum directional Lipschitz constants are not available, one can compute them numerically using the discrete function~$\disc{\cost}_x: \setg{U} \ra \Ru$. 
In particular, if these functions are \emph{convex-extensible}, it is possible to compute the range of slopes with an acceptable computational cost: Take $\lip_i^- (\disc{\cost}_x)$ (resp. $\lip_i^+(\disc{\cost}_x)$) to be the minimum finite first forward (resp. maximum finite last backward) difference of $\disc{\cost}_x$ along each dimension $i =1,2,\ldots,m$. 
If $\disc{\cost}_x: \setg{U} \ra \R$ is also \emph{real-valued}, computing the maximum and minimum directional Lipschitz constants using this method has a complexity of $\ord (XU)$.}
$\setg{V}(x)$ can then be constructed as explained in the following remark. 

\begin{Rem} [Construction of $\setg{V}(x)$ for $x \in \setg{X}$] \label{rem:constr V} 
Construct the dual grid $\setg{V}(x) = \Pi_{i=1}^{m} \setg{V}_{i}(x) \subset \R^m$ such that in each dimension~$i=1,2,\ldots,m$, 
the set~$\setg{V}_{i}(x) \subset \R$ contains at least two elements that are less (resp. greater) than $\lip_i^-(\disc{\cost}_x)$ (resp. $\lip_i^+(\disc{\cost}_x)$), so that $\co\big(\setsg{V} (x)\big) \supseteq \setc{L}(\disc{\cost}_x)$. 
This construction of $\setg{V}(x),\ x \in \setg{X}$, has a time complexity of $\ord (X)$.
\end{Rem}

The proposed numerical scheme also increases the computational cost of the extension of the d-CDP Algorithm~\ref{alg:d-CDP general} that uses this scheme. 
In this regard, notice that, for fixed grids $\setg{V}(x),\ x \in \setg{X}$, the first step of the scheme is carried out \emph{once} in a multistep implementation of the d-CDP algorithm. 
In particular, if the grids $\setg{V}(x),\ x \in \setg{X}$, are all of the same size $V$, for the $T$-step implementation of the d-CDP Algorithm~\ref{alg:d-CDP general}, which uses the scheme described above to compute $\lftc{\cost_x}$ numerically, 
\begin{itemize}
\item Step 1 introduces a \emph{one-time} computational cost of $\ord(X(U+V))$, and,
\item Step 2 increases the \emph{per iteration} computational cost of the algorithm to $\wt{\ord}(XY)$. 
\end{itemize} 
Hence, the extension of the d-CDP Algorithm~\ref{alg:d-CDP general} that computes $\lftc{\cost_x}$ numerically has a time complexity of $\wt{\ord}\big(X(U+V) + TXY \big)$ for a $T$-step value iteration problem. 

Finally,  we note that the same scheme described above can be used for numerical computation of the conjugate~$\lftc{\costu}$ of the input cost in the modified d-CDP operator $\mdcdpo$ \eqref{eq:d-CDP op separ}. 
However, since the function is now independent of the state variable, the two steps of the scheme also become independent of $x$ (see also lines~\ref{line_alg_se:const V}-\ref{line_alg_se:LLT of g} and \ref{line_alg_se:before h}-\ref{line_alg_se:h} of Algorithm~\ref{alg:d-CDP separ extended}). 
In particular, the extension of the d-CDP Algorithm~\ref{alg:d-CDP separ} that computes $\lftc{\costu}$ numerically has a time complexity of $\wt{\ord}\big(U+V + T(X+Y+Z) \big)$ for a $T$-step value iteration problem. 

\subsubsection{Numerical simulations} \label{app:ext}
We now provide the results of our numerical simulations of the extended d-CDP algorithms. 
To simplify the exposition, we consider disturbances that have finite support $\setd{W}$ of size $W$, with a given p.m.f. $p:\setd{W}\ra [0,1]$. 
The pseudo-codes of these algorithms are provided in: 
\begin{itemize}
\item[(i)] Algorithm~\ref{alg:d-CDP general extended}: multistep implementation of the extended version of Algorithm~\ref{alg:d-CDP general};
\item[(ii)] Algorithm~\ref{alg:d-CDP separ extended}: multistep implementation of the extended version of Algorithm~\ref{alg:d-CDP separ}.
\end{itemize}
We note that all the functions involved in these extended algorithms are now discrete. 
The setup of our numerical experiments is the same as the one provided in Section~\ref{sec:numerical ex}. 
However, we now consider stochastic dynamics by introducing an additive disturbance belonging to the finite set $\setd{W} = \{-0.1,0,0.1 \}^2$ with a uniform p.m.f. $p(w) = \frac{1}{9}$ for all $w \in \setd{W}$. 
Moreover, the conjugate of the (input-dependent) stage cost, although analytically available, is computed numerically, where the dual grids of the input space ($\setg{V} (x)$ in Algorithm~\ref{alg:d-CDP general extended} and $\setg{V}$ in Algorithm~\ref{alg:d-CDP separ extended}) are constructed following the guidelines of Remark~\ref{rem:constr V}. 
Let us also note that the extension of discrete cost functions $\disc{J_t}:\setg{X}\ra \R$ is also handled via LERP (in the stochastic d-DP operation~\eqref{eq:d-DP stoch}, for the expectation operations in line~\ref{line_alg_ge:expect} of Algorithm~\ref{alg:d-CDP general extended} and line~\ref{line_alg_se:expect} of Algorithm~\ref{alg:d-CDP separ extended}, and for generating greedy control actions). 
Through these numerical simulations, we compare the performance of the stochastic d-DP algorithm and the extended d-CDP algorithms for solving one hundred instances of the optimal control problem with random initial conditions, chosen uniformly from $\setc{X} = [-1,1]^2$. 
Figure~\ref{fig:perf ext} shows the results of our numerical simulations, i.e., the total running time in seconds and the average trajectory cost using greedy control actions (similar to the setup of Section~\ref{sec:numerical ex}). 
In this regard, we note the reported running times match the complexities of the corresponding algorithms for this example:   
\begin{itemize}
\item[(i)] stochastic d-DP algorithm: $\ord(TXUW)$;
\item[(ii)] d-CDP Algorithm~\ref{alg:d-CDP general extended}: $\ord \big( X(U+V) + T X(W+Y) \big)$ -- assuming all the grids $\setg{V}(x)$ are of size $V$;
\item[(iii)] d-CDP Algorithm~\ref{alg:d-CDP separ extended}: $\ord \big( U+V + T (XW+Y+Z) \big)$. 
\end{itemize}
 
%\vspace{1cm}

\begin{figure}[t]
\begin{subfigure}{.5\textwidth}
  \centering
  \includegraphics[clip, trim=0cm 0cm 0cm .2cm,width=.6\linewidth]{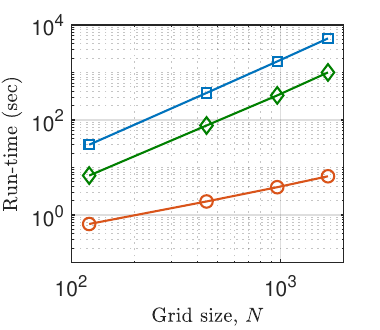}
  \caption{}
  \label{fig:complexity ext}
\end{subfigure}%
\begin{subfigure}{.5\textwidth}
  \centering
  \includegraphics[clip, trim=0cm 0cm 0cm .2cm,width=.6\linewidth]{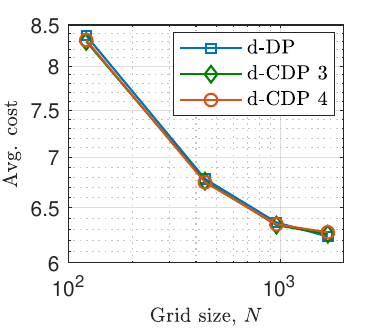}
  \caption{}
  \label{fig:cost ext}
\end{subfigure}
\caption{Performance of the stochastic d-DP algorithm and the extended d-CDP Algorithms~\ref{alg:d-CDP general extended} and \ref{alg:d-CDP separ extended} for different grid sizes ($X,Y,U,Z, V(x)=N$): (a) the total running time for solving a random problem instance; (b) the average cost of controlled trajectories for 100 random initial states.}
\label{fig:perf ext}
\end{figure}

\subsection{Echt examples} \label{app:echt ex} 
In this section, we showcase the application of the proposed d-CDP algorithms in solving the optimal control problem for two typical systems. 
In particular, we use the extended versions of these algorithms for the optimal control of the  SIR (Susceptible--Infected--Recovered) model for epidemics and a noisy inverted pendulum. 
Here, we again compare the performance of the proposed algorithms with the benchmark d-DP algorithm. 
Moreover, through these examples, we highlight some issues that can arise in the real-world application of the proposed algorithms. 

\subsubsection{SIR model} \label{app:example SIR}
We consider the application of the extended version of the d-CDP Algorithm~\ref{alg:d-CDP general} (i.e., Algorithm~\ref{alg:d-CDP general extended}) for computing the optimal vaccination plan in a simple epidemic model. 
To this end, we consider the SIR system described by \cite[Sec.~4]{Ding10}
$$
\left\{ \begin{array}{l}
s_{t+1} = s_t (1-u_t) - \alpha i_t s_t (1-u_t) \\
i_{t+1} = i_t + \alpha i_t s_t (1-u_t) - \beta  i_t  \\
r_{t+1} = r_t + u_t s_t,
\end{array} 
\right.
$$
where $s_t,i_t,r_t \geq 0$ are respectively the normalized number of susceptible, infected, and immune individuals in the population, 
and $u_t \in [0,u_{\max}]$ is the control input which can be interpreted as the proportion of the susceptibles to be vaccinated ($u_{\max}\leq 1$). 
We are interested in computing the optimal vaccination policy with linear cost $\ssum_{t=0}^{T-1}(\gamma i_t + u_t) + \gamma i_T$, over $T=3$ steps ($\gamma > 0$). 
The model parameters are the transmission rate~$\alpha = 2$, the death rate~$\beta = 0.1$, the maximum vaccination capacity $u_{\max} = 0.8$, and the cost coefficient $\gamma = 100$ (corresponding to the values in \cite[Sec.~4.2]{Ding10}). 

We now provide the formulation of this problem w.r.t. the notation of Section~\ref{sec:conj DP}. 
Note that the variable $r_t$ (number of immune individuals) can be safely ignored as it affects neither the evolution of the other two variables nor the cost to be minimized. 
Hence, we can take $x_t = (s_t,i_t) \in \R^2$ and $u_t \in \R$ as the state and input variables. 
The dynamics of the system is then described by $x_{t+1} = \dynx(x_t) + \dynu(x_t)\cdot u_t$, where  
\begin{equation*}
\dynx(s,i) =  \left[ \begin{array}{c} s - \alpha s i \\ (1-\beta)i+\alpha s i \end{array} \right], \quad \dynu(s,i) = \left[ \begin{array}{c} -s + \alpha s i \\ - \alpha s i \end{array} \right].
\end{equation*}
We consider the state constraint $x_t \in \setc{X} = [0,1]\times[0,0.5]$, and the input constraint $u_t \in \setc{U} = [0,0.8]$. 
In particular, the constraint $i_t \in [0,0.5]$ is chosen so that the feasibility condition of Assumption~\ref{As:general}-\ref{As:constr} is satisfied. 
Also, the corresponding stage and terminal costs are $\cost(s,i,u) = \gamma i + u$, and $\cost_T (s,i) = i$, respectively. 
We note that, although the conjugate of the stage cost ($\lftc{\cost_x}$) is analytically available, we use the scheme provided in Appendix~\ref{subsec:extension num conj} to compute $\lftc{\cost_x}$ numerically. 

In order to deploy the d-DP algorithm and the extended d-CDP Algorithm~\ref{alg:d-CDP general extended}, we use uniform grid-like discretizations of the state and input spaces and the their dual spaces ($\setg{X},\setg{Y} \subset \R^2$ and $\setg{U}, \setg{V}(x) \subset \R$ for $x \in \setg{X}$). 
In particular, discrete state and input spaces are such that $\co (\setg{X}) = \setc{X}$ and $\co (\setg{U}) = \setc{U}$. 
The dual grids $\setg{Y}$ and $\setg{V}(x)$ are constructed following the guidelines provided in Remarks~\ref{rem:constr Y} and \ref{rem:constr V} (with $\alpha = 0.5$). 
Let us also note that the extension of discrete cost functions $\disc{J_t}:\setg{X}\ra \R$ in d-DP is handled via LERP. 

Figure~\ref{fig:sir} depicts the computed cost $\disc{J}_0: \setg{X} \ra \R$ and control law $\disc{\mu}_0: \setg{X} \ra \setg{U}$ using the d-DP and d-CDP algorithms. 
In particular, for the d-CDP algorithm, we are reporting the simulation results for two configurations of the dual grids. 
Table~\ref{tab:sir} reports the corresponding grid sizes and the running times for solving the backward value iteration problem.
In particular, notice how the d-DP algorithm outperforms the d-CDP algorithm with the discretization scheme of configuration~1, where $X = Y$ and $U =V$. 
In this regard, we note that, in the setup of this example, the time complexity of the d-DP algorithm is of $\ord(TXU)$, while that of the d-CDP algorithm is of $\ord\big(X(U+V) + TXY) = \ord\big(XU + TX^2)$. 
Hence, what we observe is indeed expected since the number of input channels is less than the dimension of the state space. 
For such problems, we should be cautious when using the d-CDP algorithm, particularly, in choosing the sizes $Y$ and $V$ of the dual grids. 
For instance, for the problem at hand, as reported in Table~\ref{tab:sir}, we can reduce the size of the dual grids as in configuration~2 and hence reduce the running time of the d-CDP algorithm. 
However, as shown in Figure~\ref{fig:sir}, this reduction in the size of the dual grids does not affect the quality of the computed costs and hence the corresponding control laws.

\begin{figure}[t]
\begin{subfigure}{.33\textwidth}
  \centering
  \includegraphics[clip, trim=0cm 0cm 0cm 0cm,width=1\linewidth]{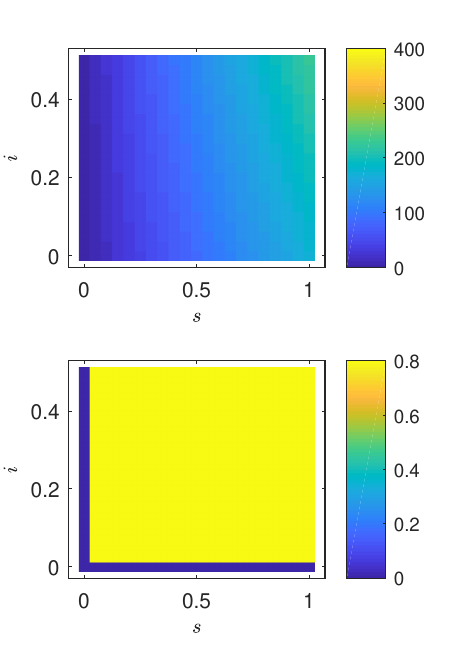}
  \caption{d-DP}
  \label{fig:J_sir_dp}
\end{subfigure}%
\begin{subfigure}{.33\textwidth}
  \centering
  \includegraphics[clip, trim=0cm 0cm 0cm 0cm,width=1\linewidth]{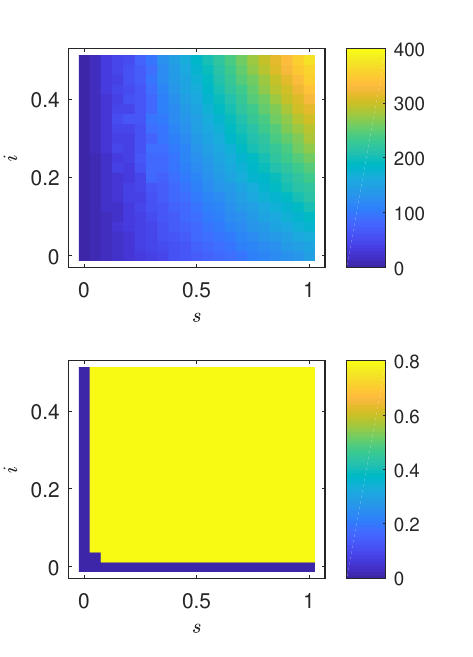}
  \caption{d-CDP (config. 1)}
  \label{fig:J_sir_cdp1}
\end{subfigure}%
\begin{subfigure}{.33\textwidth}
  \centering
  \includegraphics[clip, trim=0cm 0cm 0cm 0cm,width=1\linewidth]{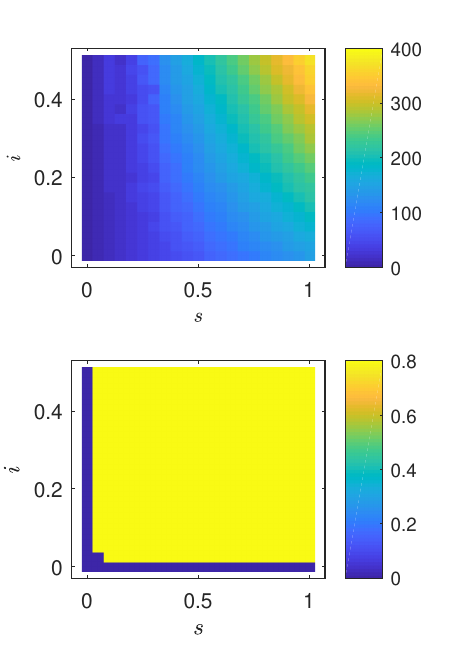}
  \caption{d-CDP (config. 2)}
  \label{fig:J_sir_cdp2}
\end{subfigure}%
\caption{Optimal control of SIR model: Cost $\disc{J}_0:\setg{X}\ra\R$ (top) and control law $\disc{\mu}_0:\setg{X}\ra\setg{U}$ (bottom).}
\label{fig:sir}
\end{figure} 

\begin{table}[t]
\caption{Optimal control of SIR model: Grid sizes and running times.}
\label{tab:sir}
%\vskip 0.15in
\begin{center}
\begin{small}
%\begin{sc}
\begin{tabular}{lcc}
\toprule
Alg. & Grid size & Running time \\
\midrule
$\mathrm{d}$-DP  & $X = 21^2,\ U = 21$ & $2.00$ sec \\
$\mathrm{d}$-CDP Alg.~\ref{alg:d-CDP general extended} (config. 1)* & $Y = 21^2,\ V = 21$ & $18.26$ sec \\
$\mathrm{d}$-CDP Alg.~\ref{alg:d-CDP general extended} (config. 2)* & $Y = 11^2,\ V = 11$ & $6.19$ sec \\
\bottomrule
\multicolumn{3}{l}{*$X$ and $U$ are the same as in $\mathrm{d}$-DP.} \\
\end{tabular}
%\end{sc}
\end{small}
\end{center}
%\vskip -0.1in
\end{table}

%\vspace{1cm}

\vskip 1cm

\subsubsection{Inverted pendulum} \label{app:example IP} 
We now consider an application of the extension of the d-CDP Algorithm~\ref{alg:d-CDP separ} (i.e., Algorithm~\ref{alg:d-CDP separ extended}) which handles additive disturbance in the dynamics. 
To this end, we consider the optimal control of a noisy inverted pendulum with quadratic costs, over a finite horizon. 
The deterministic, continuous-time dynamics of the system is described by 
$ \ddot{\theta} = \alpha \sin\theta + \beta \dot{\theta} + \gamma u$, 
where $\theta$ is the angle (with $\theta = 0$ corresponding to upward position), and $u$ is the control input \cite[Sec.~4.5.3]{Busoniu17}. 
The values of the parameters are $\alpha = 118.6445$, $\beta = -1.599$, and $\gamma = 29.5398$ (corresponding to the values of the physical parameters in \cite[Sec.~4.5.3]{Busoniu17}). 
Here, we consider the corresponding discrete-time dynamics, by using the forward Euler method with sampling time $\tau = 0.05$. 
We also introduce stochasticity by considering an additive disturbance in the dynamics. 
The discrete-time dynamics then reads as 
$
x_{t+1} = \dynx(x_{t}) + B u_t + w_t,
$ 
where $x_t = (\theta_t,\dot{\theta}_t) \in \R^2$ is the state variable (angle and angular velocity), $w_t \in \R^2$ is the disturbance, and 
\begin{equation*}
\dynx(\theta,\dot{\theta}) =  \left[ \begin{array}{c} \theta \\ \dot{\theta} \end{array} \right] + \tau \cdot \left[ \begin{array}{c} \dot{\theta} \\ \alpha \sin\theta + \beta \dot{\theta} \end{array} \right], \quad B = \left[ \begin{array}{c} 0 \\ \gamma \end{array} \right].
\end{equation*} 
We consider the state constraint 
$x_t \in \setc{X} = [-\frac{\pi}{3},\frac{\pi}{3}] \times [\pi,\pi] \subset \R^2$, 
and the input constraint
$u_t \in \setc{U} = [-3,3] \subset \R$. 
The control horizon is $T=50$, and the state, input, and terminal costs are quadratic, i.e., $\cost_{\nu}(\cdot) = \norm{\cdot}^2, \ \nu \in \{\text{s},\text{i},T\}$. 
We note that the conjugate of the input cost $\lftc{\costu}$ is analytically available, 
and given by $\lftc{\costu}(v) = \hat{u}v - \hat{u}^2, \ v \in \R$, where $\hat{u} = \max \left\{ -3, \ \min \left\{\frac{v}{2}, \ 3 \right\} \right\}$. 
Finally, we assume that the disturbances $w_t$ are i.i.d., with a uniform distribution over the finite  support $\setd{W} = \{0, \pm0.025\frac{\pi}{3}, \pm0.05\frac{\pi}{3} \} \times \{0, \pm0.025\pi, \pm0.05\pi \} \subset \R^2$ of size $W=5^2$. 

We solve the optimal control problem described above by deploying the stochastic versions of the d-DP algorithm~\eqref{eq:d-DP stoch} and the extended d-CDP Algorithm~\ref{alg:d-CDP separ extended} which handles additive disturbance in the dynamics using the method described in Appendix~\ref{subsec:extension stochastic}. 
We use uniform, grid-like discretizations $\setg{X}$ and $\setg{U}$ for the state and input spaces such that $\co (\setg{X})= [-\frac{\pi}{4},\frac{\pi}{4}]\times[-\pi,\pi]\subset \setc{X}$ and $\co (\setg{U})= \setc{U}$. 
This choice of discrete state space $\setg{X}$ particularly satisfies the feasibility condition of Assumption~\ref{As:feasible discrete}. 
(Note that the set~$\setc{X}$ however does not satisfy the feasibility condition of Assumption~\ref{As:general}-\ref{As:constr}). 
For the construction of the grids $\setg{Y}$ and $\setg{Z}$, we follow the guidelines provided in Remarks~\ref{rem:constr Y} and \ref{rem:constr Z} (with $\alpha = 2$).
We note that the extension of discrete cost functions $\disc{J_t}:\setg{X}\ra \R$ in all the algorithms is handled via \emph{nearest neighbor} (w.r.t the discrete points in $\setg{X}$). 
 
The computed cost $\disc{J}_0: \setg{X} \ra \R$ and control law $\disc{\mu}_0: \setg{X} \ra \setg{U}$ using the d-DP and d-CDP algorithms are shown in Figure~\ref{fig:ip}, 
and Table~\ref{tab:sir} reports the grid sizes and the running times for solving the backward value iteration problem. 
In particular, notice how the d-CDP algorithm has a significantly lower time requirement compared to the d-DP algorithm. 
In this regard, we note that, in the setup of this example, the time complexity of the (stochastic) d-DP algorithm is of $\ord(TXUW)$, while that of the d-CDP algorithm is of $\ord\big(T(XW+Y+Z)) = \ord (TXW)$. 
 
\begin{figure}[t]
\begin{subfigure}{.5\textwidth}
  \centering
  \includegraphics[clip, trim=0cm 0cm 0cm 0cm,width=.8\linewidth]{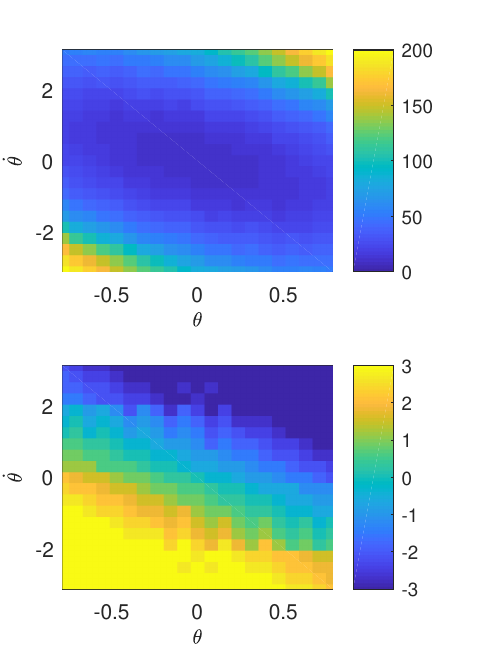}
  \caption{d-DP}
  \label{fig:J_ip_dp}
\end{subfigure}%
\begin{subfigure}{.5\textwidth}
  \centering
  \includegraphics[clip, trim=0cm 0cm 0cm 0cm,width=.8\linewidth]{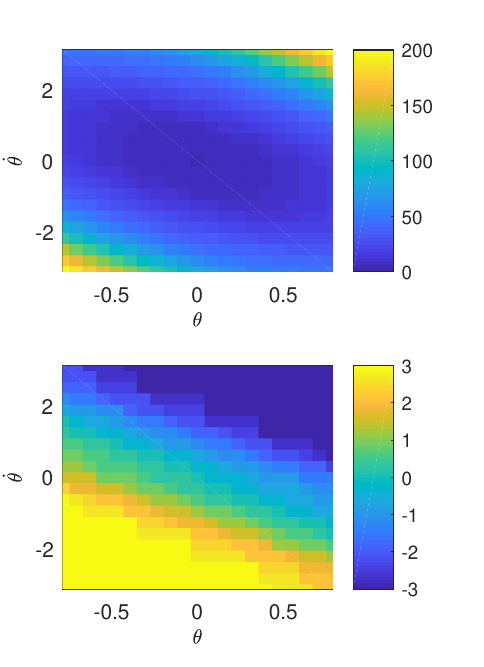}
  \caption{d-CDP}
  \label{fig:J_ip_cdp}
\end{subfigure}%
\caption{Optimal control of pendulum: Cost $\disc{J}_0:\setg{X}\ra\R$ (top) and control law $\disc{\mu}_0:\setg{X}\ra\setg{U}$ (bottom).}
\label{fig:ip}
\end{figure} 

\begin{table}[t]
\caption{Optimal control of pendulum: Grid sizes and running times.}
\label{tab:ip}
%\vskip 0.15in
\begin{center}
\begin{small}
%\begin{sc}
\begin{tabular}{lcc}
\toprule
Alg. & Grid size & Running time \\
\midrule
$\mathrm{d}$-DP  & $X = 21^2,\ U = 21$ & $281.2$ sec \\
$\mathrm{d}$-CDP Alg.~\ref{alg:d-CDP separ extended} & $X,Y,Z = 21^2$ & $9.3$ sec \\
\bottomrule
\end{tabular}
%\end{sc}
\end{small}
\end{center}
%\vskip -0.1in
\end{table}

%============================================
\section{The d-CDP MATLAB package} \label{app:matlab} 

The MATLAB package \cite{Kolari20} concerns the implementation of the two d-CDP algorithms (and their extensions) developed in this study. 
The provided codes include detailed instructions/comments on how to use them. 
Also provided are the numerical examples of Section~\ref{sec:numerical ex} and Appendix~\ref{app:num example}. 
In what follows we highlight the most important aspects of the developed package with a list of available routines. 

Recall that, in this study, we exclusively considered \emph{grid-like} discretizations of both primal and dual domains for discrete conjugate transforms. 
This allows us to use the MATLAB function \texttt{griddedInterpolant} for the LERP extensions within the d-CDP algorithms by setting the interpolation and extrapolation methods of this function to \texttt{linear}. 
However, this need not be the case in general, and the user can choose other options available in the \texttt{griddedInterpolant} routine, by modifying the corresponding parts of the provided codes; see the comments in the codes for more details. 
We also note that for the discrete conjugation (LLT), we used the MATLAB package (the \texttt{LLTd} routine and two other subroutines, specifically) provided in \cite{Lucet97} to develop an n-dimensional LLT routine via factorization (the function \texttt{LLT} in the package). 
Table~\ref{tab:MATLAB package 1} lists other routines that are available in the developed package. 
In particular, there are four high-level functions (functions (1-4) in Table~\ref{tab:MATLAB package 1}) that are developed separately for the two settings considered in this article. 
We also note that the provided implementations do not require the discretization of the state and input spaces to satisfy the state and input constraints (particularly, the feasibility condition of Assumption~\ref{As:feasible discrete}). 
Nevertheless, the function \texttt{feasibility\_check\_}$*$ ($* = 1, 2$) is developed to provide the user with a warning if that is the case. 
Finally, we note that the conjugates of four extended real-valued convex functions are also provided in the package (functions (11-14) in Table~\ref{tab:MATLAB package 1}).

\begin{table}[h]
\caption{List of routines available in the d-CDP MATLAB package.}
\label{tab:MATLAB package 1}
\begin{center}
\begin{small}
\vskip -0.1in
\begin{tabular}{p{4cm}p{11.9cm}}
\toprule
MATLAB Function & Description \\
\midrule
(1) \texttt{d\_CDP\_Alg\_}$*$ & Backward value iteration for finding costs using d-CDP \\
(2) \texttt{d\_DP\_Alg\_}$*$ & Backward value iteration for finding costs and control laws using d-DP \\
(3) \texttt{forward\_iter\_J\_}$*$ & Forward iteration for finding the control sequence for a given initial condition using costs (derived via d-DP or d-CDP) \\
(4) \texttt{forward\_iter\_Pi\_}$*$ & Forward iteration for finding the control sequence for a given initial condition using control laws (derived via d-DP) \\
%\midrule
(5) \texttt{feasibility\_check\_}$*$ & For checking if the discrete state-input space satisfies the constraints \\
(6) \texttt{eval\_func} & For discretization of an analytically available function over a given grid \\
(7) \texttt{eval\_func\_constr} & An extension of \texttt{eval\_func} that also checks given constraints \\
(8) \texttt{ext\_constr} & For extension of a discrete function while checking a given set of constraints \\
(9) \texttt{ext\_constr\_expect} & For computing expectation of a discrete function subjected to additive noise \\ 
(10) \texttt{slope\_range} & For computing the range of slopes of a convex-extensible discrete function with a grid-like domain \\
\midrule
(11) \texttt{conj\_Quad\_ball} & Conjugate of $g(u) = u\tr R u$ ($R\succ 0$) with $\dom(g) = \{u\in\R^n: \norm{u} \leq r\}$  \\
(12) \texttt{conj\_Quad\_box} & Conjugate of $g(u) = u\tr R u$ ($R\succ 0$) with $\dom(g) = \{u\in\R^n: \ul{u}_i \leq u_i \leq \ol{u}_i \}$  \\
(13) \texttt{conj\_L1\_box} & Conjugate of $g(u) = \ssum_{i=1}^n |u_i|$ with $\dom(g) = \{u\in\R^n: \ul{u}_i \leq u_i \leq \ol{u}_i \}$   \\
(14) \texttt{conj\_ExpL1\_box} & Conjugate of $g(u) = \ssum_{i=1}^n e^{|u_i|} - n$ with $\dom(g) = \{u\in\R^n: \ul{u}_i \leq u_i \leq \ol{u}_i \}$   \\
\bottomrule 
\multicolumn{2}{l}{$* = 1, 2$, corresponding to Settings~\ref{Set:prob class I} and~\ref{Set:prob class II}, respectively.} \\
\end{tabular}
%\vskip -0.15in
\end{small}
\end{center}
\end{table}

%===============================================================================
\bibliographystyle{apalike} %{siam}
\begin{small}
\bibliography{ref}
\end{small}
%===============================================================================

\end{document}